\documentclass[11pt]{amsart}

\title{KL Convergence Guarantees for Score diffusion models under minimal data assumptions}

\author{Giovanni Conforti}
\address{\'Ecole Polytechnique}
\curraddr{D\'epartement de Math\'ematiques Appliqu\'es,
Palaiseau, France.}
\email{giovanni.conforti@polytechnique.edu}
\thanks{GC acknowledges funding from the grant SPOT (ANR-20-CE40-
0014)}

\author{Alain Durmus}
\address{\'Ecole Polytechnique}
\curraddr{D\'epartement de Math\'ematiques Appliqu\'es,
Palaiseau, France.}
\email{alain.durmus@polytechnique.edu}
\thanks{AD would like
to thank the Isaac Newton Institute for Mathematical Sciences for support and hospitality during the programme
\emph{The mathematical and statistical foundation of future data-driven engineering} when work on this paper was undertaken. This work was supported by: EPSRC grant number
EP/R014604/1.}

\author{Marta Gentiloni Silveri}
\address{\'Ecole Polytechnique}
\email{marta.gentiloni-silveri@polytechnique.edu}

\newcommand{\tildes}{\tilde{s}}
\usepackage[T1]{fontenc}      
\usepackage[utf8]{inputenc}   

\usepackage{hyperref}
\usepackage{amsmath,amsthm}
\usepackage{mathtools}
\usepackage{geometry}
\usepackage{amssymb,mathrsfs} 
\usepackage{amssymb}
\usepackage{amsfonts}
\usepackage{upgreek}
\usepackage{nicefrac}
\usepackage{graphicx}
 \usepackage{color}
 \usepackage{algorithm2e}
\usepackage{stmaryrd}
\DeclareMathAlphabet{\mathpzc}{OT1}{pzc}{m}{it}
\usepackage[inline]{enumitem}
\usepackage{url}
\usepackage{graphicx} 
\usepackage{lmodern} 
\usepackage{xcolor}
\usepackage{bbm,bm}

\usepackage{ifthen}
\usepackage{xargs}
\usepackage{caption}
\usepackage{subcaption}

\usepackage{todonotes}

\hypersetup{
  colorlinks = true,
  citecolor = blue,
}

\usepackage{aliascnt}
\usepackage{cleveref}
\usepackage{autonum}
\makeatletter
\newtheorem{theorem}{Theorem}
\crefname{theorem}{theorem}{Theorems}
\Crefname{Theorem}{Theorem}{Theorems}

\newtheorem{lemma}{Lemma}
\crefname{lemma}{lemma}{lemmas}
\Crefname{Lemma}{Lemma}{Lemmas}

\newtheorem{corollary}{Corollary}
\crefname{corollary}{corollary}{corollaries}
\Crefname{Corollary}{Corollary}{Corollaries}

\newtheorem{proposition}{Proposition}
\crefname{proposition}{proposition}{propositions}
\Crefname{Proposition}{Proposition}{Propositions}

\crefname{definition}{definition}{definitions}
\Crefname{Definition}{Definition}{Definitions}

\newtheorem{remark}{Remark}
\crefname{remark}{remark}{remarks}
\Crefname{Remark}{Remark}{Remarks}

\crefname{example}{example}{examples}
\Crefname{Example}{Example}{Examples}

\crefname{figure}{figure}{figures}
\Crefname{Figure}{Figure}{Figures}

\newtheorem{assumption}{\textbf{H}\hspace{-3pt}}
\Crefname{assumption}{\textbf{H}\hspace{-3pt}}{\textbf{H}\hspace{-3pt}}
\crefname{assumption}{\textbf{H}}{\textbf{H}}

\Crefname{assumptionAO}{\textbf{AO}\hspace{-3pt}}{\textbf{AO}\hspace{-3pt}}
\crefname{assumptionAO}{\textbf{AO}}{\textbf{AO}}

\Crefname{assumptionL}{\textbf{L}\hspace{-3pt}}{\textbf{L}\hspace{-3pt}}
\crefname{assumptionL}{\textbf{L}}{\textbf{L}}

\Crefname{assumptionA}{\textbf{A}\hspace{-3pt}}{\textbf{A}\hspace{-3pt}}
\crefname{assumptionA}{\textbf{A}}{\textbf{A}}

\Crefname{assumptionG}{\textbf{G}\hspace{-3pt}}{\textbf{G}\hspace{-3pt}}
\crefname{assumptionG}{\textbf{G}}{\textbf{G}}

\Crefname{assumptionAp}{\textbf{A'}\hspace{-3pt}}{\textbf{A'}\hspace{-3pt}}
\crefname{assumptionAp}{\textbf{A'}}{\textbf{A'}}



\usepackage{version}
\excludeversion{commenta}

\def\msa{\mathsf{A}}

\def\mse{\mathsf{E}}

\def\mcf{\mathcal{F}}

\def\rset{\mathbb{R}}

\def\mrl{\mathrm{L}}
\def\rml{\mathrm{L}}
\def\rmd{\mathrm{d}}

\def\rme{\mathrm{e}}

\def\rmC{\mathrm{C}}
\def\mrC{\mathrm{C}}

\def\rmE{\mathrm{E}}

\newcommand{\Z}{\mathbb Z}

\newcommandx{\functionspace}[2][1=+]{\mathbb{F}_{#1}(#2)}

\newcommandx{\VarDeux}[3][3=]{\operatorname{Var}^{#3}_{#1}\left\{#2 \right\}}

\newcommand{\LeftEqNo}{\let\veqno\@@leqno}

\newcommand{\floor}[1]{\left\lfloor #1 \right\rfloor}

\newcommand{\floorLigne}[1]{\lfloor #1 \rfloor}

\newcommand{\N}{\ensuremath{\mathbb{N}}}

\newcommand{\PE}{\mathbb{E}}
\newcommand{\PP}{\mathbb{P}}

\newcommand{\abs}[1]{\left\vert #1 \right\vert}

\newcommandx{\Vnorm}[2][1=V]{\| #2 \|_{#1}}
\newcommandx{\VnormEq}[2][1=V]{\left\| #2 \right\|_{#1}}
\newcommandx{\norm}[2][1=]{\ifthenelse{\equal{#1}{}}{\left\Vert #2 \right\Vert}{\left\Vert #2 \right\Vert^{#1}}}
\newcommandx{\normLigne}[2][1=]{\ifthenelse{\equal{#1}{}}{\Vert #2 \Vert}{\Vert #2\Vert^{#1}}}

\newcommand{\parenthese}[1]{\left(#1 \right)}

\newcommand{\parentheseDeux}[1]{\left[ #1 \right]}

\newcommand{\ps}[2]{\left\langle#1,#2 \right\rangle}

\newcommandx\probaMarkovTilde[2][2=]
{\ifthenelse{\equal{#2}{}}{{\widetilde{\mathbb{P}}_{#1}}}{\widetilde{\mathbb{P}}_{#1}\left[ #2\right]}}

\newcommand{\expe}[1]{\PE \left[ #1 \right]}

\newcommand{\plusinfty}{+\infty}

\def\ie{\textit{i.e.}}

\def\eqsp{\;}

\newcommand{\ocint}[1]{\left(#1\right]}

\newcommand{\ccint}[1]{\left[#1\right]}

\newcommand{\ccintLigne}[1]{[#1]}

\newcommandx{\weight}[2][2=n]{\omega_{#1,#2}^N}

\def\TV{\mathrm{TV}}

\newcommandx\sequence[3][2=,3=]
{\ifthenelse{\equal{#3}{}}{\ensuremath{\{ #1_{#2}\}}}{\ensuremath{\{ #1_{#2}, \eqsp #2 \in #3 \}}}}
\newcommandx\sequenceD[3][2=,3=]
{\ifthenelse{\equal{#3}{}}{\ensuremath{\{ #1_{#2}\}}}{\ensuremath{( #1)_{ #2 \in #3} }}}

\newcommandx{\sequencen}[2][2=n\in\N]{\ensuremath{\{ #1_n, \eqsp #2 \}}}
\newcommandx\sequenceDouble[4][3=,4=]
{\ifthenelse{\equal{#3}{}}{\ensuremath{\{ (#1_{#3},#2_{#3}) \}}}{\ensuremath{\{  (#1_{#3},#2_{#3}), \eqsp #3 \in #4 \}}}}
\newcommandx{\sequencenDouble}[3][3=n\in\N]{\ensuremath{\{ (#1_{n},#2_{n}), \eqsp #3 \}}}

\def\iid{\text{i.i.d.}}

\newcommand{\opnorm}[1]{{\left\vert\kern-0.25ex\left\vert\kern-0.25ex\left\vert #1 
    \right\vert\kern-0.25ex\right\vert\kern-0.25ex\right\vert}}

\def\Id{\operatorname{Id}}

\newcommandx{\CPE}[3][1=]{{\mathbb E}_{#1}\left[\left. #2 \, \middle \vert \, #3 \right. \right]} 

\newcommandx{\CPELigne}[3][1=]{{\mathbb E}_{#1}[\left. #2 \,  \vert \, #3 \right. ]} 

\newcommandx{\CPVar}[3][1=]{\mathrm{Var}^{#3}_{#1}\left\{ #2 \right\}}
\newcommand{\CPP}[3][]
{\ifthenelse{\equal{#1}{}}{{\mathbb P}\left(\left. #2 \, \right| #3 \right)}{{\mathbb P}_{#1}\left(\left. #2 \, \right | #3 \right)}}

\newcommandx{\osc}[2][1=]{\mathrm{osc}_{#1}(#2)}

\def\Id{\operatorname{Id}}

\def\Ltwo{\mathrm{L}^2}

\def\transpose{\operatorname{T}}

\def\tT{\tilde{T}}

\newcommand\coupling[2]{\Gamma(\mu,\nu)}

\renewcommand{\geq}{\geqslant}
\renewcommand{\leq}{\leqslant}

\def\Leb{\mathrm{Leb}}

\def\Ltt{\mathtt{L}}

\def\Mtt{\mathtt{M}}

\newcommand{\wasserstein}{\mathscr{W}}

\def\wiener{\mathbb{W}}

\newcommand{\txts}{\textstyle}

\def\rmL{\mathrm{L}}

\def\bfo{\mathbf{o}}

\newcommandx{\Voi}[1][1=i]{\mathfrak{V}_{\bfo,#1}}
\newcommandx{\Vlyapc}[2][1=\bfo,2=i]{\mathfrak{V}_{#1,#2}}

\def\Wlyap{\mathfrak{W}}
\def\bfomega{\boldsymbol{\omega}}

\newcommandx{\Woi}[1][1=i]{\Wlyap_{\bfomega,#1}}
\newcommandx{\Wlyapc}[2][1=\bfomega,2=i]{\Wlyap_{#1,#2}}

\def\g{g}

\def\bfA{\mathbf{A}}

\newcommand{\app}{X^{\theta^{\star}}}
\newcommand{\appl}{V^{\theta^{\star}}}

\newcommand{\bbE}{\mathbb{E}}
 \newcommand{\foX}{\overrightarrow{X}}
 \newcommand{\foV}{\overrightarrow{V}}
 \newcommand{\baX}{\overleftarrow{X}}
 \newcommand{\baV}{\overleftarrow{V}}
 \newcommand{\mustar}{\mu_{\star}}
 \newcommand{\muprior}{\mu_0}
 
 \newcommand{\baP}{\overleftarrow{P}}
 \newcommand{\fom}{\overrightarrow{p}}

 \newcommand{\ptilde}{\Tilde{p}}

 \newcommand{\KL}{\mathrm{KL}}
 
 \newcommand{\fisher}{\mathscr{I}}

\newcommand{\thetastar}{\theta^{\star}}
\newcommand{\gY}{g}

\newcommand{\sbp}{P^{\theta^{\star}}}

\newcommand{\applaw}{p^{\theta^{\star}}}

\def\wiener{\mathbb{W}}

\def\Fr{\mathrm{Fr}}

\def\bbf{\mathbf{b}}
\def\bfSigma{\boldsymbol{\Sigma}}

\def\rmE{\mathrm{E}}

\def\Mtt{\mathtt{M}}
\def\thetas{\theta^{\star}}
\def\div{\operatorname{div}}
\def\mfrZ{\mathfrak{Z}}

\begin{document}

\maketitle

\begin{abstract}
    Diffusion models are a new class of generative models that revolve around the estimation of the score function associated with a stochastic differential equation. Subsequent to its acquisition, the approximated score function is then harnessed to simulate the corresponding time-reversal process, ultimately enabling the generation of approximate data samples. Despite their evident practical significance these models carry, a notable challenge persists in the form of a lack of comprehensive quantitative results, especially in scenarios involving non-regular scores and estimators. In almost all reported bounds in Kullback Leibler (KL) divergence, it is assumed that either the score function or its approximation is Lipschitz uniformly in time. However, this condition is very restrictive in practice or appears to be difficult to establish. 
    
    To circumvent this issue, previous works mainly focused on establishing convergence bounds in KL for an early stopped version of the diffusion model and a smoothed version of the data distribution, or assuming that the data distribution is supported on a compact manifold. These explorations have led to interesting bounds in either Wasserstein or Fortet-Mourier  metrics. However, the question remains about the relevance of such early-stopping procedure or compactness conditions. In particular, if there exist a natural and mild condition ensuring explicit and sharp convergence bounds in KL.
   
    In this article, we tackle the aforementioned limitations by focusing on score diffusion models with fixed step size stemming from the Ornstein-Uhlenbeck semigroup and its kinetic counterpart. Our study provides a rigorous analysis, yielding simple, improved and sharp convergence bounds in KL applicable to any data distribution with finite Fisher information with respect to the standard Gaussian distribution. 
\end{abstract}



\section{Introduction}
Over the past few years, deep generative models (DGMs) have emerged as a thriving research field in artificial intelligence \cite{salakhutdinov2015learning,ruthotto2021introduction} owing to their remarkable performance. Roughly speaking, generative modeling consists in
learning a map capable of generating new data instances that resemble a given set of observations, starting from a simple prior distribution, most often a standard Gaussian distribution.
 Successfully trained DGMs have the capability to approximate complex and high-dimensional probability distributions \cite{goodfellow2020generative, kingma2014stochastic, zhao2016energy, papamakarios2021normalizing} and can serve as a proxy for the data likelihood \cite{turner2019metropolis}.

The applications of DGMs are diverse and widespread, spanning statistical physics and computational chemistry \cite{carleo2019machine, brehmer2020madminer, noe2019boltzmann}, medicine \cite{sandfort2019data}, meteorology \cite{ravuri2021skilful}, reinforcement learning \cite{ho2016generative, houthooft2016curiosity},  and inverse problems \cite{song2020score}. As a result, the literature on DGMs is rapidly growing. 

Score-based diffusion generative models (SGMs) are now one of the essential pillars in generative modeling,  \cite{NEURIPS2020_4c5bcfec, dhariwal2021diffusion, song2020score, sohl2015deep}, playing a pivotal role in the most recent theoretical and practical achievements. The basic idea behind SGMs is to sample from the time-reversal of a diffusion process in order to convert noise into new data instances. In their initial step, these models construct an estimator of the score function of an ergodic (forward) diffusion process over a fixed time window $[0,T]$. After learning the score function, the second step consist in simulating the trajectories of the time-reversal of the forward process. In order to make this step computationally feasible, a suitable time-discretization of the score is introduced and the backward process is initialized at the invariant distribution of the ergodic process, whose samples are typically much easier to obtain than samples from the marginals of the forward process. The final outcome of this approximate time-reversal diffusion yields samples that are expected to be good approximations of the data distribution. 

The impressive empirical performances of SGMs have motivated the development of an intense research activity aiming at quantifying how the various sources of error (score approximation, time-discretization and initialization) affect the quality of the returned samples, thus providing with a theoretical framework that justifies the success of SGMs in applications. Specifically, there has been a growing interest in understanding the generative sampling phase (\ie, the second step described earlier) and providing theoretical guarantees on its effectiveness.
However, despite remarkable and very recent progresses \cite{chen2022sampling,chen2022improved,block2020generative,de2021diffusion,de2022convergence,lee2022convergence,lee2023convergence,wibisono2022convergence} in developing a mathematical theory for diffusion models, there is currently no comprehensive quantitative result that provides a priori bounds on the discrepancy  between the generated samples and the data distribution without relying on smoothness assumptions on the score function or its estimator, in particular a Lipschitz type condition. On the other hand, if one accepts to introduce an early stopping rule that consists in running the approximated backward process up to time $T-\delta$, some of above mentioned works have shown that, under minimal assumptions on the data, it is possible to quantify the discrepancy between the returned samples and the law of the forward process at time $\delta$, that may be regarded as a noised (smoothed) version of the data distribution: we refer to \Cref{sec:comparison} for a more thorough review of the existence literature and comparisons with the results of this article.

\subsection*{Our contribution} In this work we analyze the performances of two popular families of score based diffusion models obtained by considering as a forward process either the Ornstein-Uhlenbeck (OU) diffusion or its kinetic counterpart (kOU) under different assumptions on the data distribution. To generate the approximate time-reversal diffusion, we consider the exponential integrator Euler-Maruyama discretization scheme with constant step size, a scheme that has been widely considered in existing research on the subject. 

In our main contribution (see \Cref{theo_absolute_OU_just_fisher} and \Cref{theo_absolute_kOU}) we establish explicit, simple, and sharp bounds on the Kullback-Leibler divergence between the data distribution and the law of SGM both in the overdamped and kinetic setting. We achieve this under the sole assumptions of an $\rml^2$-score approximation error and that data distribution has finite relative Fisher information with respect to the standard Gaussian distribution. We remark that previous results that do not require the introduction of an early stopping rule and/or exponential decreasing step sizes    \cite{lee2022convergence,lee2023convergence,chen2022improved} were obtained assuming either the data distribution is supported on a bounded manifold or that the score (or its approximation) is Lipschitz regular,  uniformly over the sampling interval $[0,T]$. Note that requiring finite Fisher information amounts to a mere integrability assumption on the score of the data distribution. 
Finally, our bounds are sharp in the sense that if the data distribution is the standard Gaussian distribution, the only term appearing in our bounds corresponds to the approximation of the score function of the diffusion process.
Moreover, they match and often improve the accuracy of previously obtained bounds.

In addition we show that replacing classical assumptions on the absolute $\rml^2$-score approximation error with natural assumptions on the relative $\rml^2$-score approximation error leads to a substantial improvement of all the above mentioned results, see  \Cref{theo_relative_OU_just_fisher}, \Cref{theo_relative_kOU}. To the best of our knowledge, such an assumption is new in the literature about SGMs. As pointed out above, we refer the reader to \Cref{sec:comparison} for a detailed comparison between our findings and the existing literature on SGMs.
Our approach is characterized by the introduction of a stochastic control perspective, \ie , by the interpretation of the backward process as the solution of a stochastic control problem. The control-theoretic (or variational) interpretations of densities of diffusion processes we made use of is, by now, quite common and exploited for different purposes, spanning nonequilibrium thermodynamical systems \cite{pavon1989stochastic}, parabolic PDEs \cite{fleming1985stochastic} (\cite{fleming1985stochastic} deals also with the issue of early stopping via the so-called penalty method in the theory of optimal control), functional inequalities \cite{lehec2013representation} and variational characterizations of Langevin diffusions \cite{karatzas2022variational}. Furthermore, in the contest of GMs,  such perspective has already been successfully implemented in \cite{tzen2019theoretical} and \cite{coleman2021sampling} to propose and analyze a probabilistic generative model which share similarities with the usual SGMs that we consider in this work. In contrast to \cite{tzen2019theoretical}, our focus lies in investigating the dynamics of the \emph{relative score process},  see \eqref{eq:def_Y_t} below. We interpret this process as a solution of the adjoint equation within a stochastic maximum principle (SMP) for the aforementioned control problem. A similar discussion has been prompted also by \cite{coleman2021sampling} but to tackle the problem of statistical inference, rather than convergence. By adopting this approach and by leveraging a standard (\cite{chen2022sampling}, \cite{chen2022improved}) decomposition of the KL divergence (see \eqref{general_KL_decomposition} below), we are able to derive accurate convergence bounds for SGMs avoiding any consideration on the regularity of the derivative of the score. Moreover, we are able to steer away from any representation of the score derivative in terms of conditional covariance matrices, a key ingredient in most recent approaches \cite{chen2022improved,chen2022sampling}.
Moreover, our stochastic control approach unveils an interesting connection between the convergence analysis of SGMs and Bakry \'Emery theory \cite{bakry2014analysis}. This connection prompts a natural conjecture that our results extend to more general choices of forward processes. This includes diffusions with non-linear drifts and diffusions on Riemannian manifolds.  We leave the verification of these conjectures to future work. 

\subsection*{Notation}
Given a measurable space $(\mse, \mathcal{F})$, we denote by $\mathcal{P}(\mse)$ the set of probability measures on $\mse$. Also, given a topological space $(\mse, \tau)$, we use $\mathcal{B}(\mse)$ to denote the Borel $\sigma$-algebra on $\mse$. Given $T>0$, we denote by $\wiener_T := \mrC([0,T], \rset^d)$ the space of continuous functions from $[0,T]$ to $\rset^d$. Recall that $(\wiener_T, \norm{\cdot}_\infty)$, with $\norm{\omega}_\infty = \sup_{t \in \ccint{0,T}} \abs{\omega_t}$ for $\omega \in \wiener_T$, is a complete separable normed space often referred to as the Wiener space. 
Let $(X_t)_{t \in [0,T]}$ be a continuous stochastic process. We denote by $P_{[0,T]}^X$ the distribution induced by this process on $\wiener_T$. We denote by $\gamma^d$ the density of the standard Gaussian distribution on $\rset^d$. With abuse of notation, we identify the standard Gaussian measure with its density. Also, we denote by $\Id_d$ the identity matrix of order $d$. Given two vectors $x, y\in \rset^d$, we denote by $\norm{x}$ and $x\cdot y$ respectively the Euclidean norm of $x$ and the canonical scalar product between $x$ and $y$. Also we write $x\lesssim y$ (resp. $x \gtrsim y$) to mean $x\le C y$ (resp. $x\ge C y$) for a universal constant $C>0$. 
Given two probability measures $\mu, \nu \in \mathcal{P}(\mse)$, the relative entropy (or $\KL$-divergence) of $\mu$ with respect to $\nu$ is defined by $\KL(\mu |\nu) := \int \log (\rmd \mu /\rmd \nu) \rmd\mu$ if $\mu$ is absolutely continuous with respect to $\nu$, and $\KL(\mu |\nu) := +\infty$ otherwise. 
If $\mse=\rset^d$, the Fisher information of $\mu$ with respect to $\nu$ is defined by $\fisher(\mu |\nu) := \int \norm{\nabla\log (\rmd \mu /\rmd \nu)}^2\rmd\mu$ if $\mu$ is absolutely continuous with respect to $\nu$ and $\log \rmd \mu/\rmd \nu$ belongs to the Sobolev space of order two \cite{hajlasz1996sobolev} associated with $\nu$ , and $\fisher(\mu |\nu) := +\infty$ otherwise. We also denote by $\Mtt^2_2$ the second moment of a probability measure $\mu\in\mathcal{P}(\mathbb{R}^d)$, \ie, $\Mtt^2_2:= \int \norm{x}^2\mu(\rmd x)$. We denote by $\Pi(\mu,\nu)$ the set of couplings between $\mu$ and $\nu$, \ie , 
  $\xi\in\Pi(\mu,\nu)$ if and only if $\xi$ is a probability measure on $\rset^d\times\rset^d$ and $\xi(\msa \times \rset^d) = \mu(\msa)$ and $\xi(\rset^d\times \msa) = \nu(\msa)$ for all measurable $\msa \subseteq \rset^d$.  If $\mu$ and $\nu$ have finite second moment, the $2-$Wasserstein distance is defined by $\wasserstein_2^2 (\mu, \nu) := \inf_{\xi \in \Pi(\mu,\nu)} \int \norm{x-y}^2 \rmd \xi(x,y)$. 
  Given a matrix $\mathbf{A}\in \rset^{d\times d}$, the Frobenius norm of $\mathbf{A}$ is given by $\norm{\mathbf{A}}_{\Fr}:= \sqrt{\sum_{i,j=1}^{d} |\mathbf{A}_{ij}|^2}=\sqrt{\text{Tr}(\mathbf{A} \cdot \mathbf{A}^{\top})}$. 
For $T \geq 0$ and $F: \ccint{0,T}\times \rset^d\to \rset^d$, $f :\ccint{0,T}\times  \rset^d \to \rset$ regular enough,  we denote by $\div, \Delta$ the divergence and Laplacian operators with respect to the space variable $x$, \ie,  $\div F := \sum_{i=1}^d \partial_{x_i} F_i$, $\Delta f := \sum_{i=1}^d \partial_{x_i}^2 f$, where $\partial_{x_i}$ denotes the partial derivative with to $x_i$.

  \section{Main results}
  
  \subsection{Score generative models}
  \label{sec:score-gener-models}
  In this section, we provide a brief summary of the ideas behind the construction of SGMs based on diffusion \cite{song2020score}. Subsequently, we introduce the settings we will be working on.
  
  Denote by $\mustar \in \mathcal{P}(\mathbb{R}^d)$ the data distribution. The first building block in SGMs is to consider a $d$-dimensional ergodic diffusion on $\ccint{0,T}$, for a fixed time horizon $T>0$, that is a stochastic differential equation (SDE) of the form
\begin{equation}
  \label{eq:diffusion_forward_general}
    \rmd\foX_t=\bbf(\foX_t)\rmd t+\bfSigma \rmd B_t \eqsp, \quad t\in [0,T]\eqsp,
  \end{equation}
  where $\bbf : \rset^d \to \rset^d$ is a drift function,
  $\bfSigma \in \rset^{d \times d}$ is a fixed covariance matrix (\ie , a
  symmetric and semi-definite positive matrix) and $(B_t)_{t \geq 0}$ is a
  $d$-dimensional Brownian motion.  Under mild assumptions on $\bbf$,
  \eqref{eq:diffusion_forward_general} admits unique strong solutions
  and is associated to a Markov semigroup $(P_t)_{t \geq 0}$ with a unique stationary distribution $\muprior$ defined for any $x \in\rset^d$, $\msa \in \mathcal{B}(\rset^d)$ and
  $t >0$, by $P_t(x,\msa) = \PP(\foX_t^x \in\msa)$, where $(\foX_t^x)_{t \geq 0}$
  is the solution of \eqref{eq:diffusion_forward_general} starting
  from $x$.  Common choices of forward processes are either the Brownian motion\footnote{Note that this process does not admit a stationary distribution but $\mu_0$ in practice is replaced by a uniform distribution on an appropriate region or a Gaussian distribution with large variance \cite{song2019generative}}, where $\bbf=0$ and $\bfSigma=\Id$,  or the Ornstein-Uhlenbeck process, where $\bbf(x)=-x$ and $\bfSigma=\sqrt{2}\Id$.
  
  The second step of SGM is to initialize \eqref{eq:diffusion_forward_general} at the data distribution $\mustar$. This means setting $\foX_0 $ to have the distribution $\mustar$, and considering the family of probability
  measures  $\{\mustar P_t \, :\, t \in \ccint{0,T}\}$ with corresponding marginal time densities (with respect to the Lebesgue measure) $\{\fom_t\, :\, t\in \ccint{0,T}\}$, \ie , $\rmd \mustar P_t/\rmd \Leb = \fom_t$ for any $t \in \ccint{0,T}$, which exist under appropriate conditions on $\bbf$ and $\bfSigma$ \cite{kusuoka2010existence}.

  Remarkably, as shown in \cite{anderson1982reverse,follmer2005entropy}, \eqref{eq:diffusion_forward_general} admits a time-reversal process, in the sense that the SDE   defined by
\begin{equation}
  \label{eq:diffusion_backward_general}
    \rmd \baX_t=(-\bbf(\baX_t) + \bfSigma\bfSigma^T \nabla \log \fom_{T-t}(\baX_t)) \rmd t+\bfSigma \rmd B_t \eqsp, \quad t\in [0,T] \eqsp,
  \end{equation}
  admits a weak solution  $(\baX_t)_{t\in\ccint{0,T}}$ with initial distribution $\mustar P_T$, such that $\foX_t$ and $\baX_{T-t}$ have the same distribution, for any $t\in \ccint{0,T}$. Rigorously speaking, the Brownian motion driving \eqref{eq:diffusion_backward_general} is different from $(B_t)_{t\geq 0}$ and can be explicitly characterized \cite[Remark 2.5]{haussmann1986time}. However, since we only deal with the distribution of $(\baX_t)_{t\in\ccint{0,T}}$ (rather than its actual random trajectory), for sake of simplicity, we identify the two.
Therefore, the last step of SGM involves following the SDE \eqref{eq:diffusion_backward_general} with $\baX_0$ initialized at $\mustar P_T$, resulting in a sample from the data distribution.

When implementing these ideas in practice, three computational challenges arise:

\begin{enumerate}[label=(\alph*),wide, labelwidth=!, labelindent=0pt]
\item  One cannot obtain \iid~samples from $\mustar P_T$. As a solution, samples from the
  stationary distribution $\muprior$ of
  \eqref{eq:diffusion_forward_general} are used instead

\item The score of the forward process, $\nabla \log \fom_{T-t} (x)$, which appears in \eqref{eq:diffusion_backward_general}, is intractable. To address this, an estimator $s_{\thetastar}$ is learned based on a family of parameterized functions $\{(t,x)\mapsto s_\theta(t,x)\}_{\theta\in \Theta}$ parameterized by $\Theta$, aiming at approximating the score. The parameter $\thetastar$ is typically determined by optimizing a discretized version of the score-matching objective:
\begin{equation}
\label{score-matching-objective}
\theta \mapsto \int_{0}^{T} \mathbb{E}\left[\norm{s_\theta(t,\foX_t) - \bfSigma \bfSigma^T\nabla \log \fom_{t}(\foX_t)}^2\right] \mathrm{d} t\eqsp,
\end{equation}
such that the $\rml^2$ estimation error is minimized
\item Once approximations for $\mustar P_T$ and the score are obtained, the continuous dynamics can be simulated using various discretization schemes. A common choice is the Euler-Maruyama (EM) discretization scheme. For a choice of sequence of step sizes $\{h_k\}_{k=1}^N$, $N \geq 1$, and the corresponding time discretization $t_k= \sum_{i=1}^{k} h_i $, such that $t_N = T$, it
defines the continuous process $(X^{\rmE}_t)_{t \in\ccint{0,T}}$
  recursively on the intervals $[t_k, t_{k+1}]$ by
  $$\rmd X^{\rmE}_t=\{-\bbf(X^{\rmE}_{t_k}) + s_{\thetastar}(T-t_k,X^{\rmE}_{t_k})\} \rmd t+\bfSigma \rmd B_t \eqsp, \quad t\in [t_k, t_{k+1}]\eqsp, \quad \text{with}\quad X^{\rmE}_0 \sim \muprior \eqsp. $$
  An alternative discretization scheme considered in this work is the stochastic Euler Exponential Integrator (EI) \cite{durmus2015quantitative}. It defines a process $(\app_t)_{t \in \ccint{0,T}}$ approximating \eqref{eq:diffusion_backward_general} recursively on intervals $[t_k, t_{k+1}]$, as the solution to the SDE:
  \begin{equation}
    \label{eq:EI_v0}
    \rmd\app_t=\{-\bbf(\app_{t}) + s_{\thetastar}(T-t_k,\app_{t_k})\} \rmd t+\bfSigma \rmd B_t, \quad t\in [t_k, t_{k+1}]\eqsp, \quad \text{with}\quad  \app_0 \sim \muprior \eqsp,\end{equation}
It turns out for some choices of $\bbf$ that this SDE can be exactly solved. In particular, in the sequel, we only consider linear drift functions $\bbf$, for which it is the case. 
\end{enumerate}

\subsection{Convergence bounds for OU-based SGM}

First, we focus on the Ornstein -Uhlenbeck case by taking $\bbf(x)= -x$ and $\bfSigma=\sqrt{2}\Id$. In this case $\muprior$ is the standard Gaussian distribution denoted by $\gamma^d$, and \eqref{eq:diffusion_forward_general} and \eqref{eq:diffusion_backward_general} turn respectively into 
\begin{equation}\label{FE}
    \rmd \foX_t=-\foX_t \rmd t+\sqrt{2} \rmd B_t\eqsp, \quad t\in [0,T]\eqsp, \quad \foX_0 \sim \mustar \eqsp,
  \end{equation}
  and
  \begin{equation}\label{BE_v0}
    \rmd\baX_t=\{\baX_t + 2 \nabla \log \fom_{T-t}(\baX_t)\}\rmd t+\sqrt{2} \rmd B_t \eqsp, \quad t\in [0,T]\eqsp,\quad \baX_0 \sim \mustar P_T \eqsp.
  \end{equation}
Keeping for simplicity, the same notation as in \Cref{sec:score-gener-models}, $(P_t)_{t \in\ccint{0,T}}$ and $\{\fom_t \,:\, t \in \ccint{0,T}\}$ denote in particular the semigroup and the marginal densities associated with \eqref{FE} respectively.
  A simple observation that will play a key role in our analysis is that since $\nabla\log \gamma^d(x) = -x$, \eqref{BE_v0} can be rewritten in the following equivalent form:
\begin{equation}\label{BE}
    \rmd\baX_t=(-\baX_t + 2 \nabla \log \ptilde_{T-t}(\baX_t)) \rmd t+\sqrt{2} \rmd B_t \eqsp, \quad t\in [0,T]\eqsp,\quad \baX_0 \sim \mustar P_T \eqsp,
  \end{equation}
  where for any $(t,x)\in\ccint{0,T}\times \rset^d$, $\ptilde_t$ is the density of the law of $\foX_t$ against the Gaussian distribution, \ie ,
\begin{equation}\label{def-ptilde}
\ptilde_t(x)= \fom_t(x)/\gamma^d(x)\eqsp.
\end{equation}
This viewpoint on time-reversal has indeed already been fruitfully employed in \cite{cattiaux2021time} to give a rigorous meaning to the backward SDE \eqref{eq:diffusion_backward_general} under minimal regularity assumptions. In the context of this article, the key observation is that if we define the relative score process for $t\in[0,T]$,
\begin{equation}
Y_t = 2\nabla \log \ptilde_{T-t}(\baX_t)\eqsp,
\end{equation}
then the It\^o differential of $(Y_t)_{t\in[0,T]}$ can be computed in explicit form, see \Cref{PMP} below.

In this article, we shall consider SGMs that generate approximate trajectories of the backward process based on its representation \eqref{BE}, thus slightly deviating from the standard literature on SGMs that relies on the representation \eqref{BE_v0}. Note however that this comes at no extra computational cost. The resulting algorithm is then obtained following exactly the same numerical procedure that leads to \eqref{eq:EI_v0}. However, let us nevertheless proceed to give a full description of the algorithm considered here for the sake of clarity. To do so, we start by considering a parametric family $\{\tildes_{\theta}\}_{\theta \in\Theta}$ for the relative score function 
$(t,x) \mapsto \nabla \log \ptilde_t(x)$. Since $\nabla \log \ptilde_t(x) = x + \nabla \log \fom_t(x)$ for any $x\in\rset^d$ and $t\in[0,T]$, this family can be constructed explicitly from a family of estimators  $\{s_{\theta}\}_{\theta \in\Theta}$ as considered in \eqref{score-matching-objective}.
   Then, for a learned parameter $\thetastar$, a sequence of step sizes $\{h_k\}_{k=1}^N$, $N \geq 1$, such that $\sum_{k=1}^N h_k = T$, the resulting
  OU-based SGM is then described by: $\app_0 \sim \gamma^d$ and for
  $k\in \{0,\ldots,N-1\}$,
\begin{equation}\label{eq:SGM-OU-Based}
\rmd \app_t=(-\app_{t} + \tildes_{\thetastar}(T-t_k,\app_{t_k})) \rmd t+\sqrt{2} \rmd B_t\eqsp, \quad t\in [t_k, t_{k+1}]\eqsp,
\end{equation}
where we recall that $\{t_k\}_{k=0}^N$ are defined by $t_0=0$ and the recursion  $t_{k+1} = t_k + h_{k+1}$. Let us now proceed to the statement of our main results for the OU-based SGM. As explained above we consider two types of assumption on the $\rml^2$-score approximation error. The first one is an assumption on the absolute error.
\begin{assumption}
  \label{ass:estimation-absolute}
There exist $\varepsilon^2 >0$ and $\thetastar \in \Theta$ such that
     \begin{equation}\label{ass}
\frac{1}{T}\sum_{k=0}^{N-1} h_{k+1}\mathbb{E} \left[\norm{\tildes_{\thetastar}(T-t_k,\foX_{T-t_k})- 2\nabla \log \ptilde_{T-t_{k}}(\foX_{T-t_k})}^2\right] \le \varepsilon^2 \eqsp,
      \end{equation}
      where $(\foX_t)_{t\in[0,T]}$ is the processes defined in \eqref{FE} and $\ptilde$ is given by \eqref{def-ptilde}.
    \end{assumption}
Note that, recalling the definition of $\ptilde_t$, the condition \eqref{ass} is equivalent to the same condition on the original score function, namely if one is able to construct an estimator $s_{\thetastar}$ such that for any $k\in\{0,...,N-1\}$, 
 \begin{equation}\label{ass_v0}
         \mathbb{E} \left[\norm{s_{\thetastar}(T-t_k,\foX_{T-t_k})- 2\nabla \log \fom_{T-t_{k}}(\foX_{T-t_k})}^2\right] \le \varepsilon^2\eqsp,
\end{equation}
then this gives also an estimator $\tildes_{\theta*}$ satisfying \eqref{ass}. Assumptions of the type \eqref{ass_v0} have already been considered in most analyses of SGM; see e.g., \cite{lee2022convergence,lee2023convergence,chen2022sampling}. In addition, as shown in \cite[Appendix A]{chen2022improved}, \Cref{ass:estimation-absolute} is satisfied in some simple scenarios.

Note that in practice, we do not have access to the function
\begin{equation}
  \label{eq:12}
  \theta \mapsto \frac{1}{T}\sum_{k=0}^{N-1} h_{k+1}\mathbb{E} \left[\norm{\tildes_{\theta}(T-t_k,\foX_{T-t_k})- 2\nabla \log \ptilde_{T-t_{k}}(\foX_{T-t_k})}^2\right] \eqsp,
\end{equation}
 but only to an empirical version of this function based on \iid~samples from $\mu^{\star}$. This raises an additional complexity level in the analysis that is out of the scope of the paper. Nevertheless, one possible direction is to rely on the new developments introduced in \cite{oko2023diffusion}. In particular, under appropriate conditions and for a well-chosen class of neural networks, \cite[Theorem 4.3]{oko2023diffusion} essentially shows that the second moment of the difference between the minimizer of the empirical loss function and the true score can be bounded as $n_{\text{s}}^{\frac{-2s}{d+2s}}$, where $n_{\text{s}}$ is the number of available samples from $\mu^{\star}$ and $s$ is a parameter associated to the smoothness of the density of $\mu^{\star}$ with respect to the Lebesgue measure. Therefore, we may expect that $\varepsilon$ in \Cref{ass:estimation-absolute} is of the same order with respect to $n_{\text{s}}^{\frac{-2s}{d+2s}}$, at least.

In order to be able to compare the law at time $T$ of \eqref{eq:SGM-OU-Based} with the data distribution we require finite relative Fisher information, \ie , $\rml^2$-integrability of the score of the data distribution.
    \begin{assumption}
  \label{ass:hyp_on_mustar}
$\mustar$ is absolutely continuous with respect to the Gaussian measure $\gamma^d$ and has finite relative Fisher information against $\gamma^d$, that is 
    \begin{equation}\label{fish}
        \fisher(\mustar|\gamma^d)= \int \norm{\nabla \log \left( \frac{\rmd \mustar}{\rmd \gamma^d}\right)}^2 \rmd \mustar < +\infty\eqsp.
    \end{equation}
\end{assumption}
Note that under \Cref{ass:hyp_on_mustar}, the Kullback Leibler divergence of $\mustar$ with respect to $\gamma^d$ is finite since $\gamma^d$ satisfies a log-Sobolev inequality $\KL(\mustar | \gamma^d) \lesssim \fisher(\mustar|\gamma^d)$; see \cite[Proposition 5.5.1]{bakry2014analysis}. This implies in turn that $\mustar$ has finite second order moment applying the Donsker Varadhan representation of the Kullback Leibler divergence. As a result, we denote by
\begin{equation}
  \label{eq:defMtt_2}
  \Mtt_2^2 = \int \norm{x}^2 \rmd \mustar(x) \eqsp.
\end{equation}

Equipped with these two conditions, we now state our first result.

\begin{theorem}\label{theo_absolute_OU_just_fisher}
Let $T\geq1,h\leq1$ and  assume  \Cref{ass:estimation-absolute}-\Cref{ass:hyp_on_mustar}. Consider the EI scheme $(\app_t)_{t\in [0,T]}$ with constant step size $h>0$ defined by \eqref{eq:SGM-OU-Based}. Denoting for any $t \in [0,T]$ by $\applaw_t$ the distribution of $\app_{t}$ we have that
    \begin{equation} \label{i}
      \KL(\mustar|\applaw_T)\lesssim \rme^{-2T} \KL(\mustar|\gamma^d) + \mathbf{C}(T,\varepsilon) + h \fisher(\mustar|\gamma^d)\eqsp,
    \end{equation}
    where $\mathbf{C}(T,\varepsilon) = T \varepsilon^2$. Moreover, the bound \eqref{i} also holds if we replace the term $\KL(\mustar|\gamma^d)\rme^{-2T}$ with $ (\Mtt_2^2+d)\rme^{-T}$. In particular, choosing 
    \begin{equation}\label{TandNfinitefishabserr}
    T = (1/2)\log((\Mtt_2+d)/\varepsilon^2), \quad N= (T \mathscr{I}(\mustar|\gamma^d))/\varepsilon^2
    \end{equation}
    makes the approximation error $\tilde{O}(\varepsilon^2)$, where the notation $\tilde{O}$ indicates that logarithmic factors of $d$ and $\varepsilon$
have been dropped.
  \end{theorem}
  \begin{proof}
    The proof is postponed to \Cref{sec:proof-theor-refth}.
  \end{proof}
  
  The term $\mathbf{C}(T,\varepsilon)$ in \eqref{i} accounts for the fact that the score $(t,x) \mapsto \ptilde_t(x)$ is replaced in the discretization \eqref{eq:SGM-OU-Based}  by the estimator the score estimator $\tildes_{\thetastar}$ satisfying \Cref{ass:estimation-absolute}. 


Secondly, we consider the case of relative small $\mrl^2$ estimation error, \ie ,
\begin{assumption}\label{ass:estimation-relative}
There exist $\varepsilon^2 >0$ and $\thetastar \in \Theta$ such that, for any $k\in\{0,...,N-1\}$, 
    \begin{equation}\label{assr}
        \mathbb{E} \left[\norm{\tildes_{\thetastar}(T-t_k,\foX_{T-t_k})- 2\nabla \log \ptilde_{T-t_{k}}(\foX_{T-t_k})}^2\right]
        \le \varepsilon^2 \mathbb{E} \left[\norm{2\nabla \log \ptilde_{T-t_k}(\foX_{T-t_k})}^2\right]\eqsp,
    \end{equation}

      where $(\foX_t)_{t\in[0,T]}$ is the processes defined in \eqref{FE} and $\ptilde$ is given by \eqref{def-ptilde}.
    \end{assumption}
    Note that the above assumption could also be written in an integral (averaged) form as it is the case for \Cref{ass:estimation-absolute}. Assumption \Cref{ass:estimation-relative} appears meaningful, considering that as $t\to \infty$, the function ${\ptilde_{t}}$ converges exponentially fast to the constant function equal to $1$ in $\rml^1(\gamma^d)$, given that the OU process converges to $0$ in total variation at this rate. Therefore, it is expected that the task of learning the relative score $\nabla \log \ptilde_t$ becomes easier as $t\to \infty$, and the additional term $\mathbb{E}[\normLigne{\nabla \log \ptilde_{T-t_k}(\foX_{T-t_k})}^2]$ in the right-hand side of \eqref{assr} is introduced to account for this. In addition, in the simple case $\mustar\equiv \mathcal{N}(\mu, \Id)\eqsp,$ for some $\mu \in \rset^d\eqsp,$ as
shown below in \Cref{Mean_of_H3}, \Cref{ass:estimation-relative} holds with high probability. \\
    Under \Cref{ass:estimation-relative}, we can improve the convergence bounds stated in \Cref{theo_absolute_OU_just_fisher}:

\begin{theorem}\label{theo_relative_OU_just_fisher}
  Let $T\geq1,h\leq1$ and assume  \Cref{ass:estimation-absolute}-\Cref{ass:estimation-relative}. Consider the EI scheme $(\app_t)_{t\in [0,T]}$ with constant step size $h>0$ defined by \eqref{eq:SGM-OU-Based}. Denoting for any $t \in [0,T]$ by $\applaw_t$ the distribution of $\app_{t}$, \eqref{i} holds with 
  $\mathbf{C}(T,\varepsilon) = {\varepsilon^2\fisher(\mustar|\gamma^d)}$. 
\end{theorem}
  \begin{proof}
    The proof is postponed to \Cref{sec:proof-theor-refth}.
  \end{proof}

  Note that considering \Cref{ass:estimation-relative} allows us to mitigate the effects of the score approximation error and to relate it to the Fisher information of the data distribution. We present our third and last result: we show that employing an exponential-then-constant scheme for the step sizes, we can achieve error bounds that scale logarithmically instead then linearly in the Fisher information.
  \begin{theorem}\label{theo_exponential_then_linear}
Let $c \in \ocint{0,1/2}$ and  $T\geq1+2c$ and assume \Cref{ass:estimation-absolute}-\Cref{ass:hyp_on_mustar}. Set $\Ltt=d^{-1}\fisher(\mustar|\gamma^d)$ and assume that $\Ltt \geq 2$. Choose the constant and exponentially decreasing sequence of  step-size,  \ie, satisfying for $k < N$, $h_{k+1}= c \min\{ \max\{T-t_k,a\},1\}$\footnote{An explicit expression is provided in \eqref{step_size_dec} in the supplementary document.}, with $a \leq 1/\Ltt$. Denoting for any $t\in[0,T]$ by $p^{\thetastar}_t$ the distribution of $X^{\thetastar}_t$ we have that
 \begin{equation} \label{i-exponential_step}
      \KL(\mustar|\applaw_T)\lesssim \rme^{-2T} \KL(\mustar|\gamma^d) + \mathbf{C}(T,\varepsilon) + c [a\Ltt d+(d+\Mtt_2^2)( \log (1/a)+1)]\eqsp,
    \end{equation}
    where $\mathbf{C}(T,\varepsilon) = T \varepsilon^2$. Moreover, the bound \eqref{i-exponential_step} also holds if we replace the term $\KL(\mustar|\gamma^d)\rme^{-2T}$ with $ (\Mtt_2^2+d)\rme^{-T}$. In particular, choosing 
    \begin{equation}\label{TandNfinitefishabserr_exponential_step}
T = (1/2)\log((\Mtt_2+d)/\varepsilon^2) \eqsp, \quad c =  \frac{\varepsilon^2}{(d+\Mtt^2_2)\log \Ltt}   \eqsp, \quad a = 1/\Ltt \eqsp,
    \end{equation}
    implies that
    \begin{equation}
      \label{eq:statemeent_and_bound}
       N\lesssim (d+\Mtt_2)\log(\Ltt)(T+\log(\Ltt))/\varepsilon^2 \eqsp,
     \end{equation}
     and
    makes the approximation error $\tilde{O}(\varepsilon^2)$, where the notation $\tilde{O}$ indicates that logarithmic factors of $d,\varepsilon$
and $\Mtt^2_2$ have been dropped.
\end{theorem}
  \begin{proof}
    The proof is postponed to \Cref{appendix_exponential_case}.
  \end{proof}

\begin{corollary}\label{coro:theo_exponential_then_linear}
Let $c \in \ocint{0,1/2}$, $\delta\in (0,1/2],$ $T\geq1+2c$ and assume \Cref{ass:estimation-absolute}. Assume only that $\mustar$ has finite second order moment. Set $\Ltt=d^{-1}\fisher( \mustar P_{\delta} |\gamma^d)$ and assume that $\Ltt \geq 2$. Choose the constant and exponentially decreasing sequence of  step-size, \ie, satisfying for $k < N$, $h_{k+1}= c \min\{ \max\{T-t_k,1/L\},1\}$. Denoting for any $t\in[0,T]$ by $p^{\thetastar}_t$ the distribution of $X^{\thetastar}_t$ we have that 
 \begin{equation} \label{i-exponential_step_coro}
      \KL(\mustar P_{\delta}|\applaw_{T-\delta})\lesssim \rme^{-T}\{\Mtt_2^2 +d\} + \mathbf{C}(T,\varepsilon) + c [(d+\Mtt_2^2)( \log(d +\Mtt^2_2) + \log(\delta^{-1}) +1)]\eqsp,
    \end{equation}
    where $\mathbf{C}(T,\varepsilon) = T \varepsilon^2$. 
  \end{corollary}
  \begin{proof}
    Once observed that  \Cref{CD0N} implies $\fisher( \mustar P_{\delta} |\gamma^d) \lesssim d/\delta + \Mtt_2^2$, the proof just consists in applying \Cref{theo_exponential_then_linear} with $a= 1/\Ltt$ to the smoothed density $\mustar P_{\delta}$ instead of $\mustar$.
  \end{proof}

\subsection{Convergence bounds for the kOU-based SGM}

Second and last, following \cite{dockhorn2021score}, we deal with an other diffusion, namely with the kinetic Ornstein -Uhlenbeck (kOU) process. Compared to the OU process, the kOU process is defined by a coupled system of SDEs which involves a new variable representing the velocity process
\begin{equation}\label{FEG}
        \rmd \foX_t= \foV_t \rmd t \eqsp,\quad \rmd \foV_t=-\{\foX_t+2\foV_t\} \rmd t+2 \rmd B_t\eqsp, \quad t\in [0,T]\eqsp, \quad (\foX_0, \foV_0) \sim \mustar\otimes\gamma^d\eqsp,
      \end{equation}
      and admits $\muprior=\gamma^{2d}$ as unique stationary distribution.
Keeping for simplicity, the same notation as in \Cref{sec:score-gener-models}, $(P_t)_{t \in\ccint{0,T}}$ and $\{\fom_t \,:\, t \in \ccint{0,T}\}$ denote in particular the semigroup and the transition density associated with \eqref{FEG} respectively.
      Reasoning as before, the time reversal of \eqref{FEG} is a weak solution of the SDE: for $t \in \ccint{0,T}$,
\begin{align} \label{BEG}
&        \rmd \baX_t= -\baV_t \rmd t\eqsp,\quad \rmd \baV_t=\{\baX_t-2\baV_t+4\nabla_v\log\ptilde_{T-t}(\baX_t, \baV_t) \}\rmd t+2 \rmd B_t\eqsp,
\end{align}
where $(\baX_0, \baV_0) \sim (\mustar\otimes\gamma^d) P_T$ and $\ptilde$ is defined by \eqref{def-ptilde}.
Given an estimator $\tildes_{\thetas}$ for the score $\{\nabla_v\log\ptilde_{t}\,: \, t \in \ccint{0,T}\}$ and  a sequence of step sizes $\{h_k\}_{k=1}^N$, $N \geq 1$, such that $\sum_{k=1}^N h_k = T$, the resulting kOU-based SGM from \eqref{eq:EI_v0} that we consider is described by the following discretization scheme:  for $k\in\{0,...,N-1\}$ and $t\in [t_k, t_{k+1}]$
   \begin{equation} \label{eq:SGM-kOU-Based}
        \rmd \app_t= -\appl_t \rmd t\eqsp,\quad 
        \rmd \appl_t=\{\app_t-2\appl_t+ \tildes_{\thetas}(T-t_k,\app_{t_k},\appl_{t_k} ) \}\rmd t+2 \rmd B_t\eqsp,
      \end{equation}
      where we recall that  $t_0 = 0$, $t_{k+1} = t_k + h_{k+1}$ for $k \in \{0,\ldots,N-1\}$, and  $(\app_0 ,\appl_0) \sim \gamma^{2d}$.

In the sequel we state the kOU counterparts of \Cref{theo_absolute_OU_just_fisher,theo_relative_OU_just_fisher}. That is, we consider either the case of small absolute $\mrl^2$ estimation error, \ie , 
\begin{assumption}\label{ass:estimation-absolute-G}
    There exist $\varepsilon^2 >0$ and $\thetastar\in \rset$ such that
    \begin{equation} \label{assG}
        \frac1T \sum_{k=0}^{N-1} h_{k+1} \mathbb{E} \left[\norm{\tildes_{\thetas}(T-t_k,\foX_{T-t_k},\foV_{T-t_k})- 4\nabla_v \log \ptilde_{T-t_k}(\foX_{T-t_k}, \foV_{T-t_k})}^2\right] 
         \le \varepsilon^2 
    \end{equation}
where $(\foX_t, \foV_t)_{t\in[0,T]}$ is the processes defined in \eqref{FEG} and $\ptilde$ is given by \eqref{def-ptilde}. 
\end{assumption}
As before, we also study the case of small relative $\mrl^2$ estimation error, \ie ,
\begin{assumption} \label{ass:estimation-relative-G}
    There exists $\varepsilon^2 >0$ and $\thetastar\in \rset$ such that, for any $k\in\{0,\ldots,N-1\}$,  
    \begin{multline}\label{assGr}
      \mathbb{E} \left[\norm{\tildes_{\thetas}(T-t_k,\foX_{T-t_k},\foV_{T-t_k})- 4\nabla_v \log \ptilde_{T-t_k}(\foX_{T-t_k}, \foV_{T-t_k})}^2\right] \\
      \le \varepsilon^2 \mathbb{E} \left[\norm{\nabla_v \log \ptilde_{T-t_k}(\foX_{T-t_k}, \foV_{T-t_k})}^2\right]\eqsp,
    \end{multline}where $(\foX_t, \foV_t)_{t\in[0,T]}$ is the processes defined in \eqref{FEG} and $\ptilde$ is given by \eqref{def-ptilde}.
\end{assumption}

\begin{theorem}\label{theo_absolute_kOU}
  Let $T \ge 1$ and  assume \Cref{ass:hyp_on_mustar}-\Cref{ass:estimation-absolute-G}. Consider the EI scheme \\$(\app_t)_{t\in [0,T]}$ with constant step size $h>0$ defined by \eqref{eq:SGM-kOU-Based}. Denoting for any $t \in [0,T]$ by $\applaw_t$ the distribution of $\app_{t}$, it holds:
    \begin{equation} \label{i_kOU}
      \KL(\mustar|\applaw_T)\lesssim \rme^{-T/2} \fisher(\mustar|\gamma^{2d})  + \mathbf{C}(T,\varepsilon) + h \fisher(\mustar|\gamma^{2d})\eqsp,
    \end{equation}
    where $\mathbf{C}(T,\varepsilon) = T \varepsilon^2$.
\end{theorem}
\begin{theorem}\label{theo_relative_kOU}
Let $T \ge 1$ and  assume \Cref{ass:hyp_on_mustar}-\Cref{ass:estimation-relative-G}. Consider the EI scheme \\$(\app_t)_{t\in [0,T]}$ with constant step size $h>0$ defined by \eqref{eq:SGM-kOU-Based}. Denoting for any $t \in [0,T]$ by $\applaw_t$ the distribution of $\app_{t}$, \eqref{i_kOU} holds with 
 with 
  $\mathbf{C}(T,\varepsilon) = \varepsilon^2 \fisher(\mustar|\gamma^d)$.   

\end{theorem}

\begin{proof}
  The proof of \Cref{theo_absolute_kOU} and \Cref{theo_relative_kOU} are postponed to \Cref{sec:kou-case}. 
\end{proof}
\subsection{Related works and comparison with existing literature}\label{sec:comparison}
The great achievements of SGMs in real world applications have triggered an intense research activity aiming at providing a theoretical justification of their performances. Roughly speaking, there are two types of results available in the literature. 
The first category consists of results that require some form of smoothness on the data distribution and in turn are able to compare directly $p^{\thetastar}_T$ with $\mustar$ is some strong divergence or metric such as $\TV$ or $\KL$. The second family of results are results that make no assumptions on the data beyond some moment conditions but are not able to directly compare $p^{\thetastar}_T$ with $\mustar$. The strategy in this case is to introduce an early stopping rule, i.e., to  fix $\delta>0$ and compare $p^{\thetastar}_{T-\delta}$ with a smoothed version of $\mustar$, namely $\mustar P_{\delta}$ as done in \Cref{coro:theo_exponential_then_linear}. By bounding the distance between $\mustar P_{\delta}$ and $\mustar$, one can eventually obtain bounds in some weaker metric, typically the Fortet-Mourier metric \cite{rachev1985monge} or, under stronger assumptions on the data, Wasserstein distance. However, there seems to be no result available in the literature that compares \emph{directly} $p^{\thetastar}_T$ with $\mustar$ under the only assumption that the data are square integrable. The results (except \Cref{coro:theo_exponential_then_linear}) of this article fall in the first category; our main contribution is to show that one can obtain bounds that are at least as precise, and often better than those available under much weaker assumptions. In particular, we get rid of any form of Lipschitzianity assumption on the score, that is always present in former results, and replace it with a mere integrability assumption, the finiteness of the Fisher information. Before moving on with a precise comparison of our findings with existing results, let us give a brief overview of some of the most relevant recent contributions in the field.

\begin{enumerate}[label=$\bullet$,wide, labelwidth=!, labelindent=0pt]
\item \underline{Results without early stopping} Earlier works on convergence bounds for SGMs required strong assumption on the data distribution beyond Lipschitzianity of the score such as a dissipativity condition \cite{block2020generative}, the manifold hypothesis \cite{de2022convergence} or $\rmL^{\infty}$-bounds on the score approximation \cite{de2021diffusion} and often failed to exhibit convergence bounds with polynomial complexity. The work \cite{de2022convergence} obtains convergence guarantees in $1$-Wasserstein distance assuming that the data distribution satisfies the so-called manifold hypothesis. Imposing that the data distribution satisfies a logarithmic Sobolev inequality \cite{lee2022convergence,wibisono2022convergence} obtain bounds with polynomial complexity. Subsequently, in the paper \cite{chen2022sampling} the authors managed to drop the hypothesis that the data distribution satisfy a functional inequality and to include the kinetic Ornstein-Uhlenbeck in the analysis by assuming weak $L^2$-bounds on the score approximation and Lipschitzianity of the score. Moreover, this work introduces a Girsanov change of measure argument that we shall systematically employ in this work (see also \cite{liu2022let} for similar ideas). The recent work \cite{chen2022improved} obtains results for constant step size discretization that are comparable to those in \cite{chen2022sampling}. Moreover, their Theorem 2.5 improves on the results of \cite{chen2022sampling} by showing that, using an exponentially decreasing then linear step size one can obtain bounds whose complexity is logarithmic in the Lipschitz constant of the score. In addition, Lipschitzianity of the score is assumed there only at the intial time, and not for the whole trajectory. However, the dimensional dependence of these bounds is quadratic instead of linear.  Finally, let us mention the recent preprint \cite{pedrotti2023improved} where the authors study a model that is different from all the ones mentioned above in that the dynamics is deterministic and does not contain a stochastic term. Therefore results of this paper, though of clear interest, are not directly comparable to the ones discussed above. We note however, that (one-sided) Lipschitzianity of the score is a key assumption even in this work.
\item \underline{Results with early stopping}
Using an early stopping rule, \cite{chen2022sampling,lee2023convergence} are able to treat any data distribution with bounded support. An improvement over these results is obtained in \cite{chen2022improved} where, adopting an exponentially decreasing step size and using a refined estimate on the short time behavior of the derivative of the score, the authors are able to cover any distribution with finite second moment. 
The work \cite{benton2023linear}, appearing roughly at the same time than the present work on Arxiv, obtains a result encompassed in our \Cref{coro:theo_exponential_then_linear}, for which we only assume finite second order moment of the data distribution. They do not consider neither bounds without early stopping or SGMs based on the kinetic Langevin diffusion.
\end{enumerate}
\subsubsection{Contribution of this work}
\begin{enumerate}[label=$\bullet$,wide, labelwidth=!, labelindent=0pt]
\item \underline{Results for OU}
Bounds for $\KL(\mu^{\star}|p^{\thetastar}_T)$ for constant step size have recently been obtained, in several works \cite{chen2022sampling,lee2022convergence,lee2023convergence,chen2022improved}. At the moment of writing and to the best of our knowledge, the article \cite{chen2022improved} is the current state of the art; results under weaker assumptions (just $\Mtt_2^2<+\infty$) exist but require to introduce an early stopping rule and cannot be expressed in terms of $\KL(\mu^{\star}|p^{\thetastar}_T)$. In the following table we offer a synthetic comparison between Theorem 2.1 in \cite{chen2022improved} and our Theorem \ref{theo_absolute_OU_just_fisher}.
\begin{table}[h]
\centering
    \caption{Bounds on $\KL(\mu^*|p^{\thetastar}_T)$ for OU with constant step-size.}\par\medskip
\scalebox{0.9}{\begin{tabular}{|c|c|c|}
	\hline 
	Assumptions & Related & Error\\ on the data & References &  bound \\
	\hline
 \Cref{ass:estimation-absolute}  &  & \\ $\Mtt_2^2<+\infty$ & \cite[Theo 2.1]{chen2022improved}  & $(\Mtt^2_2+d)\rme^{-T}+T\varepsilon^2 +dhL^2T$ \\ $\nabla\log \fom_t\;$ $L-$ Lipschitz  & & \\
  \hline
    \Cref{ass:estimation-absolute} & &  \\  & Theorem \ref{theo_absolute_OU_just_fisher} & $(\Mtt_2^2+d)\rme^{-T}+T\varepsilon^2+h(dL+\Mtt^2_2)$ \\ $\fisher(\mustar|\gamma^d)\leq dL+\Mtt_2^2$  & &\\
   \hline
      \Cref{ass:estimation-relative} & &  \\  & Theorem \ref{theo_relative_OU_just_fisher} & $(\Mtt_2^2+d)\rme^{-T}+(\varepsilon^2+h)(dL+\Mtt_2^2)$ \\ $\fisher(\mustar|\gamma^d)\leq dL+\Mtt_2^2$  & &\\
   \hline
\end{tabular}}
\end{table}

\begin{table}[h]
\centering
    \caption{Number of iterations to get $\KL(\mu^*|p^{\thetastar}_T) \lesssim \varepsilon^2$ for OU with constant and exponentially decreasing step sizes.}\par\medskip
\begin{tabular}{|c|c|c|}
	\hline 
	Assumptions & Related & Steps to get\\ on the data & References & $\tilde{O}(\varepsilon^2)$ error \\
	\hline
 \Cref{ass:estimation-absolute}  &  & \\ $\Mtt_2^2<+\infty$ & \cite[Theo 2.5]{chen2022improved}  & $d^2\log^2(L)/\varepsilon^2$ \\ $\nabla\log \fom_0\;$ $L-$ Lipschitz  &\\
  \hline
    \Cref{ass:estimation-absolute} & &  \\  & Theorem \ref{theo_exponential_then_linear} & $(d+\Mtt_2)\log^2(L)/\varepsilon^2$ \\  $\fisher(\mustar|\gamma^d)\leq Ld+\Mtt^2_2$ & &\\
   \hline
\end{tabular}
\end{table}

\begin{table}[h]
\centering
    \caption{Bounds on $\KL(\mu^*P_{\delta}|p^{\thetastar}_{T-\delta})$ for OU  with constant and exponentially decreasing step sizes.}\par\medskip
\begin{tabular}{|c|c|c|}
	\hline 
	Assumptions & Related & Error\\ on the data & References &  bound \\
	\hline
  \Cref{ass:estimation-absolute}  &  & \\ $\Mtt_2^2<+\infty$ & \cite[Theo 2.2]{chen2022improved}  & $(\Mtt^2_2+d)\rme^{-T}+T\varepsilon^2 +d^2(T+\log\delta^{-1})^2/N$ \\ 
  \hline
  \Cref{ass:estimation-absolute} & &  \\  & \Cref{coro:theo_exponential_then_linear}  & $\rme^{-T}\{\Mtt_2^2 +d\} + T\varepsilon^2 $\\
$\Mtt_2^2<+\infty$& &    $+ c [(d+\Mtt_2^2)( \log(d +\Mtt^2) + \log(\delta^{-1}) +1)]$  \\
   \hline
\end{tabular}
\label{fig:comp_early_stopp}
\end{table}

For results with constant step size, when comparing assumptions, we observe that those of \cite[Theorem 2.1]{chen2022improved} or \cite[Theorem 2]{chen2022sampling} are considerably stronger than the condition $\fisher(\mustar|\gamma^d)\leq dL+\Mtt_2^2$ reported in the table above. Indeed, if $\nabla \log \mustar$ is $L$-Lipschitz, using integration by parts we obtain
\begin{equation}
\fisher(\mustar|\gamma^d) \lesssim \int \norm{\nabla\log\mustar}^2 \rmd \mustar + \Mtt_2^2 = -\int \Delta \log \mustar \rmd\mustar+\Mtt_2^2 \leq dL+\Mtt_2^2 \eqsp.
\end{equation}
When comparing the error bounds on $\KL(\mustar|p^{\thetastar}_T)$, the first two terms coincide. In the third term, that corresponds to the discretization error, the bound provided by Theorem \ref{theo_absolute_OU_just_fisher} is independent of $T$ and depends linearly in $L$ instead of quadratically. The fourth additional in \ref{theo_absolute_OU_just_fisher} is not relevant, since typically $\Mtt_2^2 \lesssim d$. Thus, Theorem \ref{theo_absolute_OU_just_fisher} is an improvement over \cite[Theorem 1]{chen2022improved}.

For results with exponential-then-constant step size, we are able to improve the bound of \cite[Theorem 2.5]{chen2022improved} in that the dependence of the dimension is linear and not quadratic. Note that we have an extra dependence on $\Mtt^2_2$ here; in the typical situation where $\Mtt^2_2 \lesssim d$ this is of no harm, and we retain the linear dependence on the dimension. Moreover, the hypothesis of our Theorem \ref{theo_exponential_then_linear} are much weaker, as we have already explained.

Finally, if $\Mtt_2^2 \leq d$, \Cref{coro:theo_exponential_then_linear} improve upon \cite[Theorem 2.2]{chen2022improved}
since while the first two terms are the same (see \Cref{fig:comp_early_stopp}), the third term is only linear with respect to the dimension and does not depend on $T$.

\item\underline{Results for kinetic OU.}  \cite{chen2022sampling} analyzed the convergence properties of SGMs based on the kOU process. 
\begin{table}[h]
\centering
    \caption{Bounds on $\KL(\mu^*|p^{\thetastar}_T)$ for kOU with constant step-size.}\par\medskip
\begin{tabular}{|c|c|c|}
	\hline 
	Assumptions & Related & Error\\ on the data & References &  bound \\
	\hline
  \Cref{ass:estimation-absolute}  &  & \\ $\fisher(\mustar|\gamma^d)<+\infty$ & \cite[Theo 6]{chen2022sampling}  & $(\KL(\mustar|\gamma^d)+\fisher(\mustar|\gamma^d))\rme^{-2cT}$ \\
  & &$ +\varepsilon^2 T+L^2Th(d+\Mtt_2^2h)$ \\ $\nabla_v\log \fom_t\;$ $L-$ Lipschitz  &  & \\
  \hline
    \Cref{ass:estimation-absolute} & &  \\  & Theorem \ref{theo_absolute_kOU} & $ \fisher(\mustar|\gamma^d)\rme^{-T} +\varepsilon^2 T + h \fisher(\mustar|\gamma^d)$ \\ $\fisher(\mustar|\gamma^d)<+\infty$  & &\\
   \hline
    \Cref{ass:estimation-absolute} & &  \\  & Theorem \ref{theo_relative_kOU} & $ \fisher(\mustar|\gamma^d)\rme^{-T} +(\varepsilon^2  + h) \fisher(\mustar|\gamma^d)$ \\ $\fisher(\mustar|\gamma^d)<+\infty$  & &\\
   \hline
\end{tabular}
\end{table}
Note that the bounds of \cite[Theorem 6]{chen2022sampling} are expressed through the total variation distance between $\mustar$ and $p^{\thetastar}_T$. However, since the bound in total variation is obtained from a bound in relative entropy through Pinsker's inequality (see Theorem 15 and the proof of Theorem 6 in \cite{chen2022sampling}), we prefer to report them in their entropic formulation to enable a clear comparison with our findings. In terms of assumptions, clearly those of Theorem \ref{theo_absolute_kOU} are considerably weaker as there is no requirement on the Lipschitzianity of the score. In contrast with the OU case, we cannot directly show that the Lipschianity of $\nabla_v\log p_t$ directly implies a bound on the Fisher information. However, it is likely that the term $h\fisher(\mustar|\gamma^d) $ is smaller than $\varepsilon^2 T+L^2Th(d+\Mtt_2^2h)$ for example because the former does not depend on the time horizon $T$. Finally, we remark that the constant $c$ appearing in the bound of \cite[Theorem 6]{chen2022sampling} is not made explicit by the authors.
\end{enumerate}
\section{Proofs}
\subsection{OU case}
We begin by justifying the smoothness of the map 
\begin{equation}\label{mapm}
    (0, T]\times \rset^d \ni (t,x) \mapsto \fom_t (x) \in \rset_+\eqsp, 
  \end{equation}
with $\fom_t$ density of the forward process defined in \eqref{FE},
\begin{proposition}\label{mapm_smooth}
  Assume that the data distribution  $\mustar$  is absolutely continuous with respect to the Lebesgue measure and denote by $\fom_0$ its density.    The map defined in \eqref{mapm}  is positive and solution of the Fokker-Planck equation on $\ocint{0,T}\times \rset^d$: for $(t,x) \in (0,T]\times \rset^d$,
  \begin{equation}
    \label{eq:1}
            \partial_t \fom_t(x) -\div (x  \fom_t) -\Delta \fom_t(x) =0 \eqsp,
          \end{equation}
          where $\div,\Delta$ are the divergence and Laplacian operator with respect to the space variable $x$, respectively. 
In addition, it belongs to $\mrC^{1,2}((0,T]\times \rset^d)$, \ie , for any $t \in (0,T)$, $x \mapsto \fom_t(x)$ is twice continuously differentiable and for any $x \in \rset^d$, $t\mapsto \fom_t(x)$ is continuously differentiable on $(0,T]$.
\end{proposition}
The proof of this well known result is reported in \Cref{sec:proof-crefm11} for the readers'convenience.
Using \Cref{mapm_smooth}, by a simple computation, we get that $\ptilde_{t}(x)$ is a classical solution on $(0,T]\times\mathbb{R}^d$ of
\begin{equation}\label{eq:forward_kolmogorov}
\Delta \ptilde_{t} - \langle \nabla \ptilde_t(x),x\rangle = \partial_t \ptilde_{t}(x)   \eqsp, \quad (t,x) \in \ocint{0,T}\times \rset^d\eqsp.
\end{equation}

Our proof strategy revolves around the study of what we call the \emph{relative score process} $(Y_t)_{t \in[0,T]}$, defined by
  \begin{equation}
    \label{eq:def_Y_t}
Y_t :=    2\nabla \log \ptilde_{T-t}(\baX_t) \eqsp,
\end{equation}
where $(\baX_t)_{t\in[0,T]}$ is the backward process defined by \eqref{BE}. $(Y_t)_{t\in[0,T]}$ is naturally connected to the problem under consideration: indeed the SDE \eqref{BE} writes
\begin{equation}
\rmd\baX_t=\{-\baX_t + Y_t\}\rmd t+\sqrt{2} \rmd B_t \eqsp, \quad t\in [0,T]\eqsp,\quad \baX_0 \sim \mustar P_T \eqsp.
\end{equation}

In the next proposition we show that the relative score process satisfies a SDE, that gains a natural interpretation as the adjoint equation in the stochastic maximum principle (SMP)\cite{krylov2008controlled,yong1999stochastic}. In particular, we shall use this equation to study properties of the score, such as its $\Ltwo$-norm along the backward process defined for $t \in \ccint{0,T}$ by
\begin{equation}
  \label{eq:def_g}
\gY(t):=\mathbb{E}[\norm{Y_t}^2] \eqsp.
\end{equation}

\begin{proposition}\label{PMP}
  Let $\delta$ be an arbitrarily fixed positive constant.
  Then it holds\begin{equation}\label{HJB}
    \rmd Y_t=Y_t \rmd t+ \sqrt{2} Z_t \rmd B_t\eqsp, \quad t\in [0,T-\delta]\eqsp,
  \end{equation}
  where  $  Z_t=2 \nabla ^2 \log \ptilde_{T-t}(\baX_t)$.
If moreover $\mustar$ has finite eighth-order moment then for any $0\le s\le t\le T-\delta$ ,
\begin{equation}\label{up2}
   \gY(t)-\gY(s) = \int_{s}^{t} \left(2\gY(r) + 2\mathbb{E}\left[\norm{Z_r}_{\Fr}^2\right]\right) \rmd r\eqsp.
\end{equation}
\end{proposition}

\begin{remark} \label{ptilde_smooth}
    For any fixed $\delta>0$, $(Z_t)_{t\in [0, T-\delta]}$ is well-defined because \Cref{mapm_smooth}
implies that, for any $\delta>0$, $(t,x) \mapsto \Tilde{p}_{T-t}(x)\in\mrC^{1,2}([0, T-\delta]\times \rset^d)$. 
\end{remark}
\begin{proof}[Proof of  \Cref{PMP}:] 
  Dividing by $\ptilde_{T-t}$ in \eqref{eq:forward_kolmogorov} and using the fact that
  \begin{equation}
  \Delta \log \ptilde_{T-t} = \div \left(\frac{\nabla \ptilde_{T-t}}{\ptilde_{T-t}}\right) = \frac{\Delta \ptilde_{T-t}}{\ptilde_{T-t}}- \norm{\frac{\nabla \ptilde_{T-t}}{\ptilde_{T-t}}}^2=\frac{\Delta \ptilde_{T-t}}{\ptilde_{T-t}} - \norm{\nabla \log \ptilde_{T-t}}^2\eqsp,\label{eq:3}
\end{equation}
we get that $(\Phi_t)_{t \in \ccint{0,T-\delta}}:= ( \log \ptilde_{T-t})_{t \in \ccint{0,T-\delta}}$ solves the Hamilton-Jacobi-Bellman equation
\begin{equation}\label{almost}
    \partial_t \Phi_t(x) + \Delta \Phi_t(x) + \norm{\nabla \Phi_t}^2(x) - \ps{x}{ \nabla \Phi_t(x)}=0\quad \text{on} \quad [0, T-\delta] \times \rset^d\eqsp.
\end{equation}
Since $Y_t=2\nabla \Phi_t (\baX_t)$, the result is now a straightforward consequence of \Cref{mapm_smooth}, It\^o formula and \eqref{almost}: 
\begin{equation}
\begin{split}
\rmd Y_t&=2\{\partial_t \nabla \Phi_t (\baX_t) +\Delta \nabla \Phi_t (\baX_t)+\nabla^2 \Phi_t (\baX_t)  (-\baX_t + 2\nabla\Phi_t(\baX_t)) \}\rmd t  \\
&+ 2\sqrt{2} \nabla^2 \Phi_t (\baX_t)\rmd B_t\\
&= 2\{\nabla (\partial_t \Phi_t +\Delta\Phi_t+\norm{\nabla\Phi_t}^2)(\baX_t) - \nabla^2 \Phi_t (\baX_t)  \baX_t \}\rmd t+\sqrt{2} Z_t \rmd B_t\\
&\stackrel{\eqref{almost}}{=} 2\nabla \Phi_t(\baX_t)\rmd t +\sqrt{2} Z_t \rmd B_t= Y_t \rmd t +\sqrt{2} Z_t \rmd B_t \eqsp, \quad t\in [0,T-\delta]\eqsp.
\end{split}
\end{equation}But then, using once again It\^o formula, we get 
\begin{equation}\label{eq:corrector}\rmd \norm{Y_t}^2= (2\norm{Y_t}^2+2\norm{Z_t}^2_{\Fr}) \rmd t + 2\sqrt{2} Y_t^{\transpose}  Z_t \rmd B_t, \quad t\in [0, T-\delta]\eqsp.\end{equation}
We now use the following lemma to show that, under the additional assumption of finite eighth-order moment for $\mustar$, 
the process $(\int_{0}^{t} Y_s^{\transpose} Z_s \rmd B_s)_{s\in [0,T-\delta]}$ is a true martingale and not just a local martingale. The proof of this technical result is deferred to \Cref{sec:proof-crefnull_e}.
\begin{lemma}\label{null_E}
    Assume that $\mustar$ has finite eighth-order moment. Then it holds \\$\PE[\int_0^{T-\delta} \norm{Y_s^{\transpose} Z_s}^2\rmd s] < \plusinfty$. In particular, $(\int_{0}^{t} Y_s^{\transpose} Z_s \rmd B_s)_{t\in [0,T-\delta]}$ is a martingale.
\end{lemma}
But then, if $\mustar$ has finite eighth-order moment because of the above lemma for any $0\le s \le t \le T-\delta$ it holds
$\txts \mathbb{E}[\int_s^tY_r^{\transpose} Z_r\rmd B_r]=0\eqsp,$
 whence, taking expectations in \eqref{eq:corrector},
 \begin{align}
     \gY(t)&= \mathbb{E}\left[\norm{Y_{t}}^2\right]= \mathbb{E}\left[\norm{Y_{s}}^2\right] + \int_{s}^{t} \mathbb{E}\left[2\norm{Y_{r}}^2+2\norm{Z_{r}}_{\Fr}^2\right] \rmd r\\
     &=\gY(s)+\int_{s}^{t} \left(2\gY(r) + 2\mathbb{E}\left[\norm{Z_r}_{\Fr}^2\right]\right) \rmd r\eqsp,
\end{align}
which concludes the proof. 
\end{proof}
\subsubsection{Proof of Theorem \ref{theo_absolute_OU_just_fisher} and \ref{theo_relative_OU_just_fisher}}
\label{sec:proof-theor-refth}
The following result states that $\gY$ increases exponentially. The reader accustomed with Bakry-\'Emery theory will recognize that its statement is equivalent to the exponential decay of the Fisher information along the OU semigroup. Note that \Cref{gYup} would be a direct consequence of \eqref{up2} if we could take directly  $\delta=0$ in \Cref{PMP}. In this sense, \Cref{gYup} is just a technical extension of \Cref{PMP}. For this reason, its proof is deferred to \Cref{sec:proof-crefgyup}.
\begin{proposition}\label{gYup}
Assume \Cref{ass:hyp_on_mustar}. Then, for any $0\le s\le t\le T$, it holds $\gY (s) \le \rme^{-2(t-s)}\gY (t) \eqsp.$ 
\end{proposition}
\begin{proof}[Proof of Theorem \ref{theo_absolute_OU_just_fisher}:]

  Let $({\tilde{B}}_t)_{t \geq 0}$be a $d$-dimensional Brownian motion on the complete universe $(\Omega,\mcf,\PP)$ and denote by $(\mcf_t^{\tilde{B}})_{t \geq 0}$ the corresponding generated filtration. Given a time horizon $\tT>0$ and a constant $\sigma>0$, consider two diffusion-type processes $(X^b_t)_{t \in \ccintLigne{0,\tT}}$, $(X^c_t)_{t \in \ccintLigne{0,\tT}}$ satisfying $\rmd X^b_t= b_t((X^b_u)_{u \in \ccintLigne{0,\tT}}) \rmd t +  \sigma \rmd {\tilde{B}}_t$ and $\rmd X^c_t= c_t((X^c_u)_{u \in \ccintLigne{0,\tT}}) \rmd t + \sigma \rmd {\tilde{B}}_t$. Here, the two processes $(b_t((X^c_u)_{u \in \ccintLigne{0,\tT}}))_{t \in \ccintLigne{0,\tT}}$ and $(c_t((X^c_u)_{u \in \ccintLigne{0,\tT}}))_{t \in \ccintLigne{0,\tT}}$ are supposed to be $(\mcf_t^{\tilde{B}})_{t \geq 0}$-adapted  and to satisfy
    \begin{equation}
      \label{eq:hyp_girsanov}
      \PP(\int_{0}^{\tT} [\Vert  b_t((X^b_u)_{u \in \ccintLigne{0,\tT}})\Vert^2 + \Vert  c_t((X^c_u)_{u \in \ccintLigne{0,\tT}})\Vert^2] \rmd t < \infty) = 1 \eqsp.
    \end{equation}
    Then it holds \begin{equation}\label{general_KL_decomposition}
      \KL(P^c_{[0,\tT]}|P^b_{[0,\tT]})=\KL(P^c_0|P^b_0)+ \frac{1}{2\sigma^2}\mathbb{E}\left[\int_{0}^{\tT} \norm{b-c}^2(X_{[0,t]}^c) \, \rmd t\right]\eqsp,
  \end{equation} where $P^c_{[0,\tT]}$ and $P^b_{[0,\tT]}$ are the distribution of $(X^c_u)_{u\in[0,\tT]}$ and $(X^b_u)_{u\in[0,\tT]}$ respectively. Indeed, as a direct consequence of \cite[Theorem 7.6, Theorem 7.7]{liptser2013statistics} plus a standard change of variable argument, we get that $P^b_{\ccintLigne{0,\tT}}$ and $P^c_{\ccintLigne{0,\tT}}$ are equivalent. More precisely, we get that $P^c_{\ccintLigne{0,\tT}}$ with respect to $P^b_{\ccintLigne{0,\tT}}$ satisfies $\PP$-almost surely
  \begin{multline}
    \label{eq:2}
    \frac{\rmd P^c_{\ccintLigne{0,\tT}}}{\rmd P^b_{\ccintLigne{0,\tT}}}((X_u^b)_{u \in \ccintLigne{0,\tT}}) = \exp\left\{ \frac{1}{2\sigma^2} \int_{0}^{\tT} \Big[\Vert  b_{t}((X_u^b)_{u \in \ccintLigne{0,\tT}})\Vert^2 - \Vert  c_{t}((X_u^b)_{u \in \ccintLigne{0,\tT}})\Vert^2\Big] \rmd t  \right.
    \\ \left. + \frac{1}{\sigma^2}\int_0^{\tT} \ps{c_{t}((X_u^b)_{u \in \ccintLigne{0,\tT}}) - b_{t}((X_u^b)_{u \in \ccintLigne{0,\tT}})}{\rmd X_t^b} \right\} \eqsp.
  \end{multline}
  By taking the logarithm and using the definition of KL divergence, we obtain \eqref{general_KL_decomposition}.
   
But then, by taking $\tT=T-\delta$ with $\delta$ such that $0<\delta<h$ and $(X^b_t)_{t\in[0,\tT]}\equiv(\app_t)_{t\in [0,T-\delta]}$ and $(X^c_t)_{t\in [0,\tT]}\equiv(\baX_t)_{t\in [0,T-\delta]}$ (the hypothesis of \cite[Theorem 7.6, Theorem 7.7]{liptser2013statistics}, in particular \eqref{eq:hyp_girsanov}, are satisfied because of \Cref{gYup}, \Cref{ass:hyp_on_mustar} and \Cref{ass:estimation-absolute}), we get
 \begin{equation}\label{KL_decomposition}
     \begin{split}
 &\KL(\baP_{[0,T-\delta]}|\sbp_{[0,T-\delta]})\\ &=\KL(\fom_T|\gamma^d)
  + \frac{1}{4}\int_{T-h}^{T-\delta} \mathbb{E} \left[ \norm{\tilde{s}_{\thetastar}(h, \baX_{T-h}) - 2\nabla \log \ptilde_{T-t}(\baX_t)}^2\right] \rmd t\\
 &+ \sum_{k=0}^{N-2} \frac{1}{4}\int_{kh}^{(k+1)h} \mathbb{E} \left[ \norm{\tilde{s}_{\thetastar}(T-kh, \baX_{kh}) - 2\nabla \log \ptilde_{T-t}(\baX_t)}^2\right] \rmd t
\eqsp,
\end{split}
\end{equation}where $\baP_{[0,T-\delta]}$ (resp. $\sbp_{[0,T-\delta]}$) is the law of $\baX_{[0,T-\delta]}$ (resp. $\app_{[0,T-\delta]}$) and $N=T/h$ is the number of iterations. With an application of the triangular inequality we can write
\begin{equation}\label{eq:three_errors_no_details}
\KL(\baP_{[0,T-\delta]}|\sbp_{[0,T-\delta]}) \lesssim E_{1}+E_{2}+E_{3}
\end{equation}
with 
\begin{equation}\label{three_errors}
    \begin{split}
    E_1&=\KL(\fom_T|\gamma^d) \eqsp, \quad
    E_2= h \sum_{k=0}^{N-1}\mathbb{E} \left[ \norm{\tilde{s}_{\thetastar}(T-hk, \baX_{hk}) - 2\nabla \log \ptilde_{T-hk}(\baX_{hk})}^2\right] \\
E_3&=\sum_{k=0}^{N-2} \int_{kh}^{(k+1)h} \mathbb{E} \left[ \norm{Y_t - Y_{kh} }^2\right] \rmd t +\int_{T-h}^{T-\delta} \mathbb{E} \left[ \norm{Y_t - Y_{T-h}}^2\right] \rmd t.
    \end{split}
\end{equation}
 
As a consequence of the logarithmic Sobolev inequality for $\gamma^d$ (see \cite[Theorem 5.2.1 and Proposition 5.5.1]{bakry2014analysis}), the first error term can be bounded as follows
\begin{equation}\label{first_term}
 E_{1} \lesssim \KL(\fom_T|\gamma^d) \lesssim \rme^{-2T} \KL(\mustar | \gamma^d)\eqsp.
\end{equation}
The second term can be bounded directly using \Cref{ass:estimation-absolute}. We get 
\begin{equation} \label{secondb}
\begin{split}
E_2 \lesssim  T \varepsilon^2\eqsp.
\end{split}
\end{equation}
We are then left with the task of bounding $E_3$. This is done in the following Lemma whose proof will be shortly given.
\begin{lemma} \label{lastmin_lemma}
For any $\delta >0$,    it holds $E_3 \lesssim  h\{g(T)-g(0)\} \le h \fisher(\mustar|\gamma^d)\eqsp.$
\end{lemma}
Plugging  \eqref{first_term}-\eqref{secondb} and the conclusion of  \Cref{lastmin_lemma} into \eqref{eq:three_errors_no_details} and exploiting the data processing inequality, (see \cite[Lemma 1.6] {nutz2021introduction}), we get
\begin{equation}
    \KL(\mustar P_\delta|\applaw_{T-\delta})\leq  \KL(\baP_{[0,T-\delta]}|\sbp_{[0,T-\delta]})\lesssim \rme^{-2T} \KL(\mustar | \gamma^d) + T \varepsilon^2 + h \fisher(\mustar|\gamma^d)\eqsp.
  \end{equation}

  Now, because of the continuity of $(\foX_t)_{t\in [0,T]}$ and $(\app_t)_{t\in [0,T]}$, as $\delta$ goes to zero, we have almost sure convergence, hence (using \cite[Proposition 1.5]{baldiD}) weak convergence, of $\foX_\delta$ and $\app_{T-\delta}$ to $\foX_0$ and $\app_T$ respectively. Equation \eqref{i} then follows directly from the joint lower-semicontinuity of relative entropy \cite[Theorem 19]{van2014renyi}. In order to show that the bound \eqref{i} also holds if we replace the term $\KL(\mustar  P_T|\gamma^d)$ with $\rme^{-T} (\Mtt_2^2+d)$, it suffices to recall the well known bound 
$    \KL(\mustar P_T|\gamma^d)\lesssim (\Mtt_2^2+d)\rme^{-T}$, see \cite[Lemma C.4]{chen2022improved} for a proof.

\begin{proof}[Proof of \Cref{lastmin_lemma}]
    Consider the sequence of probability distributions $(\mustar^n)_{n\geq 1}$ with density functions $(\fom_0^n)_{n\geq 1}$ defined by 
    \begin{equation}\label{sequence}
       \fom_0^n (x) =  \rme^{-\norm{x}^2/n}\fom_0(x)/ \mfrZ_n \eqsp, \mfrZ_n = \int \rme^{-\norm{x'}^2/n}\fom_0(x') \rmd x'\eqsp.
    \end{equation}
    Then, for any $n \in \mathbb{N}$, $\;\mustar^n$ has finite eighth-order moment. 
      
Let $\delta >0$.   Now, for any $n \geq 1$ and $t \in \ccint{0,T-\delta}$, define $Y^n_t:=2\nabla\log \ptilde^n_{T-t}(\baX^n_t)$, $\;Z^n_t:=2\nabla^2\log \ptilde^n_{T-t}(\baX^n_t)$ and $\gY^n:=\mathbb{E}[\norm{Y^n_t}^2]$, where $\ptilde^n:=\fom^n/\gamma^d$ and $(\baX^n_t)_{t \in \ccint{0,T-\delta}}$ is the time-reversal of the Ornstein-Uhlenbeck process defined in \eqref{FE} with initial distribution $\mustar^n$.  Then it holds
\begin{lemma}\label{technical_lemma}
    For $\Leb$-almost every $x \in\mathbb{R}^d$ it holds
    \begin{equation} \label{lim_RHS}
        \fom_0^n(x) \to \fom_0(x)\eqsp,\quad \nabla \log\fom_0^n(x) \to \nabla\log\fom_0(x)\quad \text{and}\quad \fisher(\mustar^n|\gamma^d)\to \fisher(\mustar|\gamma^d)\eqsp.
      \end{equation}
      Moreover, it holds
      
        \begin{equation}
  \label{eq:conv_Y_n_Y}
(Y^n_t)_{t \in \ccint{0,T-\delta}} \text{ converges in distribution to } (Y_t)_{t \in \ccint{0,T-\delta}} \eqsp, \text{as $n\to \plusinfty$} \eqsp.
\end{equation}
 
\end{lemma}
We postpone to \Cref{app_technical_lemma} the proof of it.

Supposing this to be true, if we define
\begin{equation}
    E_3^n:=  \int_{T-h}^{T-\delta} \mathbb{E} \left[ \norm{Y^n_t-Y^n_{T-h}}^2\right] \rmd t + \sum_{k=0}^{T/h-2} \int_{kh}^{(k+1)h} \mathbb{E} \left[ \norm{Y^n_t-Y^n_{kh}}^2\right] \rmd t\eqsp,
    \end{equation}we obtain, as a consequence of \cite[Exercise 3.2.4]{durrett2019probability} that 
    \begin{equation}\label{lim_LHS}
        E_3\le \liminf_{n\to +\infty} E_3^n\eqsp.
    \end{equation}
    On the other side, using \Cref{PMP} and Young inequality, for $t\in [kh, T-\delta]$ and $k\in \{0,..., N-1\}$, it holds
    \begin{equation}
    \mathbb{E}\left[\norm{Y^n_t-Y^n_{kh}} ^2\right] \lesssim \mathbb{E}\left[\norm{ \int_{kh}^{t} Y^n_s \rmd s}^2\right] + \mathbb{E}\left[\norm{\int_{kh}^{t} Z^n_s \rmd B_s}^2\right]\eqsp.\label{eq:4}
  \end{equation}
Applying firstly Jensen inequality to the first term and It\^o isometry to the second one and eventually invoking \Cref{PMP}, we obtain for $t\in [kh, (k+1)h]$ and $k\in \{0, ..., N-2\}$, as well as for $t\in [kh, T-\delta]$ and $k=N-1$,
\begin{align}
\mathbb{E}\left[\norm{Y^n_t-Y^n_{kh}}^2\right] & \lesssim |t-kh| \;\mathbb{E}\left[\int_{kh}^{t}\norm{Y^n_s}^2 \rmd s\right]+ \mathbb{E}\left[\int_{kh}^{t}\norm{Z^n_s}_{\Fr}^2 \rmd s\right] \\
&\lesssim h \; \mathbb{E}\left[\int_{kh}^{t}\norm{Y^n_s}^2 \rmd s\right]+ \mathbb{E}\left[\int_{kh}^{t}\norm{Z^n_s}_{\Fr}^2 \rmd s\right] \\
\label{eq:need_ineq_later_dec_step}
& \lesssim  \mathbb{E}\left[\int_{kh}^{(k+1)h}\norm{Y^n_s}^2 + \norm{Z^n_s}_{\Fr}^2 \rmd s\right] \stackrel{\eqref{up2}}{\lesssim} g^n((k+1)h)-g^n(kh) \eqsp .
\end{align}
This bound, \eqref{lim_LHS} and \eqref{lim_RHS}
imply the thesis:
\begin{equation}
\begin{split}
E_3& \lesssim\liminf_{n\to + \infty} h\sum_{k=0}^{N-1} (\gY^n((k+1)h)-\gY^n(kh))  \lesssim \liminf_{n\to + \infty} h g^n(T)\\
&= \liminf_{n\to + \infty} h \fisher(\mustar^n|\gamma^d)=  h\fisher(\mustar|\gamma^d) \eqsp.
\end{split}
\end{equation}
\end{proof}

\end{proof}
\begin{proof}[Proof of Theorem \ref{theo_relative_OU_just_fisher}:] The proof of this Theorem is almost identical to the one of Theorem \ref{theo_absolute_OU_just_fisher}. 
The only difference is how the error term $E_2$ in \eqref{three_errors} is dealt with. In this case we have
\begin{equation}
\begin{split}
E_2&=h\sum_{k=0}^{N-1}  \mathbb{E} \left[\norm{\tilde{s}_{\thetastar}(T-kh,\baX_{kh})- 2\nabla \log \ptilde_{T-kh}(\baX_{kh})}^2\right] \\
&\stackrel{\Cref{ass:estimation-relative}}{\leq} \varepsilon^2h\sum_{k=0}^{N-1}\mathbb{E} \left[\norm{\nabla \log \ptilde_{T-kh}(\baX_{kh})}^2\right] \lesssim 
\varepsilon^2h\sum_{k=0}^{N-1}\gY(kh)\\
&\stackrel{\text{Prop.} \ref{gYup}}{\lesssim} \varepsilon^2h\fisher(\mustar|\gamma^d) \sum_{k=0}^{N-1} \rme^{-2(T-kh)} = \varepsilon^2h\fisher(\mustar|\gamma^d)  \rme^{-2T}\frac{\rme^{2T}-1}{\rme^{2h}-1}\lesssim \varepsilon^2\fisher(\mustar|\gamma^d)\eqsp.
\end{split}
\end{equation}
\end{proof}

\subsubsection{Proof of \Cref{theo_exponential_then_linear}}
\label{sec:proof-theor-refth_2}
The proof of \Cref{theo_exponential_then_linear} is postponed to \Cref{appendix_exponential_case}.

\subsection{kOU case}
\label{sec:kou-case}
As in the OU  setting, we first justify smoothness of the map 
\begin{equation}\label{mapmG}
    (0, T]\times \rset^{2d} \ni (t, x, v) \mapsto \fom_t (x,v) \in \rset_+\eqsp, 
\end{equation}
\begin{proposition}\label{mapmG_smooth}
    The map defined in \eqref{mapmG} is smooth, \ie , $\mrC^{1,2}((0,T]\times\rset^{2d})$.
\end{proposition}
The proof of the above proposition is given in \Cref{sec:proof-crefm-1}.
As a consequence of \Cref{mapmG_smooth}, a simple calculation shows that that $(\ptilde_t)_{t\in[0,T]}$ is a  classical solution of the Kolmogorov equation
\begin{equation}\label{eq:kinetic_kolmogorov}
    \partial_t \ptilde_t(x,v)-2\Delta_v\ptilde_t(x,v) -\ps{v}{\nabla_x \ptilde_t(x,v)} + \ps{(x-2v)}{\nabla_v \ptilde_t(x,v)} =0 \eqsp.
\end{equation}
In the current setting, the study of the relative score process alone does not contain enough information to provide satisfactory upper bounds for $\KL(\mustar|\gamma^d)$. Following the pathwise interpretation given in \cite{chiarini2021entropic} of ideas and concepts put forward in \cite{cedric2009hypocoercivity}, it pays off to consider the pair $(Y^v_t,Y^x_t)_{t \in[0,T])}$ defined by
  \begin{equation}
    \label{eq:def_Y_tG}
Y^v_t:=4\nabla_v \log \ptilde_{T-t}(\baX_t, \baV_t)\eqsp,\quad  Y^x_t:=4\nabla_x \log \ptilde_{T-t}(\baX_t, \baV_t)\eqsp,
\end{equation}
where  $(\baX_t, \baV_t)_{t\in[0,T]}$ is the process defined by \eqref{BEG}.
Indeed, we will show at \Cref{gup} below that the function $\g :[0,T) \to \rset_+$  defined by
\begin{equation}
  \label{eq:def_g_kinetic}
  g(t):= \mathbb{E}\left[\norm{Y^v_t}^2\right]+ \mathbb{E}\left[\norm{Y^v_t+ Y^x_t}^2\right] \eqsp.
\end{equation}
plays the same role of the the function $\gY$ studied in the previous section. 
\begin{proposition}\label{PMPG}
   Let $\delta$ be an arbitrarily fixed positive constant. Then for $t\in [0,T-\delta]$ it holds 
    \begin{equation}\label{HJBG}
    \rmd Y^x_t=-Y^v_t \rmd t+2Z^{vx}_t \cdot\rmd  B_t \eqsp, \quad \rmd Y^v_t=(2Y^v_t+ Y^x_t) \rmd t+ 2 Z^{vv}_t \cdot \rmd B_t \eqsp,
\end{equation}
where for all $t\in[0,T)$
\begin{equation}
    Z^{vx}_t := 4\nabla_v\nabla_x\log\ptilde_{T-t}(\baX_t, \baV_t), \quad Z^{vv}_t:=4 \nabla_v ^2 \log \ptilde_{T-t}(\baX_t, \baV_t)\eqsp.
\end{equation}
If moreover $\mustar$ has finite eighth-order moment then for any $0\le s\le t\le T-\delta$ it holds
\begin{equation}\label{up2G}
   \g(t)- \g(s)\ge \int_{s}^{t}\left( \g(r)+ \mathbb{E}\left[\norm{Z^{vv}_r}_{\Fr}^2\right]\right) \rmd r \eqsp.
\end{equation}
\end{proposition}
\begin{remark}\label{ptilde-smoothG} For any fixed $\delta>0$, $(Z^{vv}_t)_{t \in \ccint{0,T-\delta}}$ and  $(Z^{vx}_t)_{t \in \ccint{0,T-\delta}}$ (as well as $(Y^v_t)_{t\in [0, T-\delta]}$ and $(Y^x_t)_{t\in [0, T-\delta]}$) is well-defined because \Cref{mapmG_smooth}
implies that, for any $\delta>0$,  $\ptilde_{T-t}(x, v)\in \mrC^{1,2}([0, T-\delta]\times \rset^{2d})$.
\end{remark}
\begin{proof}[Proof of \Cref{PMPG}:] 
Fix arbitrarily $\delta>0$. 
Dividing by $\ptilde_{T-t}$ in \eqref{eq:kinetic_kolmogorov} and using the fact that $$\Delta_v \log \ptilde_{T-t} = \frac{\Delta_v \ptilde_{T-t}}{\ptilde_{T-t}} - \norm{\nabla_v \log \ptilde_{T-t}}^2,$$ 
we get that $(\Phi_t)_{t\in[0,T-\delta]}:= ( \log \ptilde_{T-t})_{t\in[0,T-\delta]}$ is a classical solution of the Hamilton-Jacobi-Bellman equation: for $(x,v) \in \rset^{2d}$ and $t \in \ccint{0,T-\delta}$,
\begin{equation}\label{almostG}
    \partial_t \Phi_t(x,v)-  v\cdot \nabla_x\Phi_t(x,v) + (x-2v)\cdot\nabla_v\Phi_t(x,v) +2 \Delta_v\Phi_t(x,v) +2 \norm{\nabla_v\Phi_t}^2(x,v)=0 \eqsp.
\end{equation}
Since $Y^v_t=4\nabla_v \Phi_t (\baX_t, \baV_t)$ and $Y^x_t=4\nabla_x \Phi_t (\baX_t, \baV_t)\eqsp,$ \eqref{HJBG} is now a consequence of \Cref{mapmG_smooth}, It\^o formula, \eqref{BEG} and \eqref{almostG}. 
\begin{commenta}Indeed, using the SDE solved by the backward process \eqref{BEG} we find, setting $\bbf^*(x,v)=(-v,x-2v)$ and $Y_t=(Y^x_t,Y^v_t)$, and using \eqref{almostG} 
\begin{equation}
    \begin{split}
      &\rmd Y_t\\
      &= 4 \{\partial_t \nabla\Phi_t + 2 \Delta_v \nabla \Phi_t+\nabla^2\Phi_t\cdot \bbf^* +4\nabla_v\nabla\Phi_t\cdot\nabla_v\Phi_t \}(\baX_t, \baV_t)\rmd t\\
      &+ 8\nabla_v\nabla\Phi_t (\baX_t, \baV_t)\cdot\rmd B_t\\
        &= 4 \{\nabla \Big(\partial_t  \Phi_t + 2 \Delta_v  \Phi_t+2\norm{\nabla_v \Phi_t}^2\Big) (\baX_t, \baV_t) +\nabla^2 \Phi_t\cdot \bbf^*(\baX_t,\baV_t)\}\rmd t +2 \begin{bmatrix} Z^{vx}_t\\Z^{vv}_t\end{bmatrix} \cdot\rmd B_t \\
       &\stackrel{\eqref{almostG}}{=}  4\{-\nabla(\nabla \Phi_t\cdot \bbf^*)+\nabla^2\Phi_t\cdot \bbf^*\}(\baX_t,\baV_t)\rmd t +2 \begin{bmatrix} Z^{vx}_t\\Z^{vv}_t\end{bmatrix} \cdot\rmd B_t \\
        &= -4\nabla\bbf^*\cdot\nabla\Phi_t(\baX_t,\baV_t)\rmd t +2 \begin{bmatrix} Z^{vx}_t\\Z^{vv}_t\end{bmatrix} \cdot\rmd B_t= \begin{bmatrix}
            0 & 1\\
            -1 & 2
        \end{bmatrix}\cdot Y_t\rmd t +2 \begin{bmatrix} Z^{vx}_t\\Z^{vv}_t\end{bmatrix} \cdot\rmd B_t \eqsp,
    \end{split}
\end{equation}
which proves \eqref{HJBG}.
\end{commenta}

Let us now proceed to show \eqref{up2G}.To this aim, we combine It\^o formula with \eqref{HJBG} to obtain
\begin{equation}
        \rmd \norm{Y^v_t}^2= \{4\norm{Y^v_t}^2+ 4\norm{Z^{vv}_t}_{\Fr}^2+ 2 \ps{Y^v_t}{ Y^x_t}\}\rmd t + 4 (Y^v_t)^{\transpose} Z^{vv}_t \cdot \rmd B_t\eqsp, \quad t\in [0, T-\delta]\eqsp,
\end{equation}and 
\begin{equation}
\begin{split}
    \rmd \norm{Y^v_t+Y^x_t}^2&= \left\{2 \norm{Y^v_t+Y^x_t}^2 + 4\norm{Z^{vv}_t + Z^{vx}_t}_{\Fr}^2\right\}\rmd t \\
    & \quad +4(Y^v_t+Y^x_t)^{\transpose} (Z^{vv}_t + Z^{vx}_t) \cdot \rmd B_t\eqsp, \quad t\in [0, T-\delta]\eqsp.
\end{split}
\end{equation}
Hence, setting for $t\in[0,T-\delta],$
$H_t:= 4 Y^v_t \cdot Z^{vv}_t +4(Y^v_t+Y^x_t)\cdot (Z^{vv}_t + Z^{vx}_t)\eqsp,$  we get that, for $t\in [0, T-\delta]$, 
\begin{equation}
    \begin{split}
      &\rmd \left(\norm{Y^v_t}^2 + \norm{Y^v_t+Y^x_t}^2\right)\\
      &\ge  \{4\norm{Y^v_t}^2+ 2 Y^v_t\cdot Y^x_t+ 2 \norm{Y^v_t+Y^x_t}^2+4\norm{Z^{vv}_t}_{\Fr}^2\}\rmd t +  H_t\cdot \rmd B_t\eqsp\\
    & =\{2\norm{Y^v_t}^2 +2 Y^v_t\cdot(Y^v_t+ Y^x_t)
    + 2 \norm{Y^v_t+Y^x_t}^2+4\norm{Z^{vv}_r}_{\Fr}^2\}\rmd t +  H_t \cdot \rmd B_t\\
    &\geq \{\norm{Y^v_t}^2 +  \norm{Y^v_t+Y^x_t}^2+4\norm{Z^{vv}_r}_{\Fr}^2\}\rmd t +  H_t \cdot \rmd B_t
    \end{split}
\end{equation}
If the process $(\int_{0}^{t}  H_s\rmd B_s)_{t\in [0,T-\delta]}$ is a true martingale, the the desired conclusion follows taking expectation on both sides in the above inequality. The next Lemma, whose proof is deferred to \Cref{sec:proof-crefgup}, takes care of this technical point.
\begin{lemma}\label{kinetic_eight_moment}
    Assume that $\mustar$ has finite eighth-order moment. Then it holds \\$\int_{0}^{T-\delta} \mathbb{E}[\norm{H_s}^2]\rmd s < \plusinfty$.  In particular, $(\int_{0}^{t} H_s \rmd B_s)_{t\in [0,T-\delta]}$ is a martingale.
\end{lemma}
\end{proof}
\subsubsection{Proof of Theorem \ref{theo_absolute_kOU} and \ref{theo_relative_kOU}} 
As in the OU case, under \Cref{ass:hyp_on_mustar}, it makes sense to consider the natural extensions of \eqref{eq:def_Y_tG} and \eqref{eq:def_g}, respectively $(Y^v_t)_{t\in [0,T]},$\\$ (Y^x_t)_{t\in [0,T]}$ and $\g : [0,T] \rightarrow \rset_+\eqsp,$ to the all time interval $[0,T]$. Also, as in the OU case, we can prove a monotonicity property for $g$. Note that \Cref{gup} would be a consequence of \Cref{PMPG} if we were allowed to take $\delta=0$ in this result. It is therefore a technical result whose proof we defer to \Cref{sec:proof-crefgup}.
\begin{proposition}\label{gup}
  Assume \Cref{ass:hyp_on_mustar} and let $g$ be as in \eqref{eq:def_g}. Then for any $0\le s \le t\le T$ it holds
  \begin{equation}\label{up1G}
    g (s) \lesssim \rme^{-(t-s)/2} g(t)  \eqsp.
  \end{equation}
\end{proposition}
The proofs of both Theorem \ref{theo_absolute_kOU} and \ref{theo_relative_kOU} have a substantial overlap with those given for Theorem \ref{theo_absolute_OU_just_fisher} and \ref{theo_relative_OU_just_fisher}. Therefore, in order to avoid repetitions, we focus on the few relevant sections where we have to argue differently.

\begin{proof}[Proof of Theorem \ref{theo_absolute_kOU}:] Fix $\delta$ such that $0<\delta<h$.

As in Theorem \ref{theo_absolute_OU_just_fisher} that, if we denote by $\baP_{[0,T-\delta]}$ and $\sbp_{[0,T-\delta]}$ respectively the laws of $(\baX_t, \baV_t)_{t\in[0,T-\delta]}$ and $(\app_t, \appl_t)_{t\in[0,T-\delta]}$, then we have
\begin{equation}
   \KL(\baP_{[0, T-\delta]}|\sbp_{[0, T-\delta]})\lesssim E_1+E_2 +E_3
\end{equation}
where 
$
E_1 \lesssim \KL(\mustar P_{T}|\gamma^{2d})\eqsp,
$
$E_2\lesssim \varepsilon^2 T$ thanks to \Cref{ass:estimation-absolute-G}, and 
\begin{equation}
E_3\lesssim \sum_{k=0}^{N-2} \int_{kh}^{(k+1)h} \mathbb{E} \left[ \norm{Y^v_t-Y^v_{kh}}^2\right] \rmd t + \int_{T-h}^{T-\delta} \mathbb{E} \left[ \norm{Y^v_t-Y^v_{T-h}}^2\right] \rmd t \eqsp.
\end{equation}
To bound $E_1$, we use the logarithmic Sobolev inequality for $\gamma^{2d}$ and then \Cref{gup} to obtain 
\begin{equation}
E_1 \lesssim \fisher(\mustar P_T|\gamma^{2d}) =\mathbb{E}[\norm{Y^x_0}^2+\norm{Y^v_0}^2]\lesssim g(0)\leq \rme^{-T/2}g(T)=\rme^{-T/2}\fisher(\mustar|\gamma^d) \eqsp,
\end{equation}
where to obtain the last inequality we used that $Y^v_T\equiv 0$ by construction. Let us now bound $E_3$ under the extra assumption that $\mustar$ admits a moment of order eight. Arguing as in Lemma \ref{lastmin_lemma}, we obtain that for all $k$ and all $t\in[kh,(k+1)h]$
\begin{equation}\label{eq:E3bound_kOU}
\begin{split}
\bbE[\norm{Y^v_t-Y^v_{kh}}^2] &\stackrel{\eqref{HJBG}}{\lesssim} \int_{kh}^{(k+1)h} \mathbb{E}[ \norm{Y^x_t}^2+\norm{Y^v_t}^2+ 2\norm{Z^{vv}_t}^2_{\Fr}] \rmd t \\
&\stackrel{\eqref{up2G}}{\lesssim} g((k+1)h)-g(kh) \eqsp.
\end{split}
\end{equation}
The validity \eqref{eq:E3bound_kOU} in the case when $\mustar$ does not admit a moment of order eight is shown with an approximation procedure exactly as in Lemma \ref{lastmin_lemma}.
But then, following closely the proof \Cref{lastmin_lemma}, since $Y^v_T\equiv 0$ by construction, we find
\begin{equation}
E_{3} \lesssim g(T)h \lesssim \mathbb{E}[\norm{Y^x_T}^2]h\lesssim\fisher(\mustar|\gamma^d)h.
\end{equation}

The proof then concludes as the one of Theorem \ref{theo_absolute_OU_just_fisher}.
\end{proof}
\begin{proof}[Proof of Theorem \ref{theo_relative_kOU}:] As in the OU case, the proof of this Theorem is almost identical to the one of Theorem \ref{theo_absolute_kOU}. The unique passage in the proof of Theorem \ref{theo_absolute_kOU} to be reconsidered is the one concerning the bound on the term $E_2$. Here, it takes the form
\begin{equation}
E_2=h\sum_{k=0}^{N-1}  \mathbb{E} \left[\norm{\tilde{s}_\theta(T-kh,\baX_{kh}, \baV_{kh})- 4\nabla_v \log \ptilde_{T-kh}(\baX_{kh}, \baV_{kh})}^2\right] \eqsp.
\end{equation}
Thanks to \Cref{ass:estimation-relative-G}, we immediately find 
\begin{equation}
\begin{split}
E_2&\lesssim \varepsilon^2h\sum_{k=0}^{N-1} \mathbb{E} \left[\norm{Y^v_{kh}}^2\right] \lesssim  \varepsilon^2h\sum_{k=0}^{N-1} g(kh) \stackrel{\text{Prop.} \ref{gup}}{\lesssim }\varepsilon^2h\fisher(\mustar|\gamma^d)\sum_{k=0}^{N-1} \rme^{-(T-kh)}.
\end{split}
\end{equation}
At this point, the proof can be completed in the same way as the proof of Theorem \ref{theo_relative_OU_just_fisher}.
\end{proof}
\appendix
\section{Technical results}

\subsection{Proof of \Cref{mapm_smooth}}
\label{sec:proof-crefm11}
  Recall that given $t\in [0,T]$ and $x,y \in \rset^d$ the transition density $q_t(x,y)$ associated to the Ornstein-Uhlenbeck semigroup of \eqref{FE} is given by
\begin{equation}\label{trans_dens_OU}
q_t(x,y)=\frac{1}{(2\pi (1-\rme^{-2t}))^{d/2}}\exp\parenthese{-\frac{\norm{y-\mathrm{e}^{-t}x}^2}{2(1-\rme^{-2t})}}\eqsp.
\end{equation}
  But then, if we fix $\delta>0$, $\fom_\delta(y)=\int \fom_0 (x) q_\delta (x,y) \rmd x$ is a continuous function of $\rset^d$. Indeed, $y \mapsto \fom_0(x)  q_\delta (x,y)$ is almost everywhere continuous and upper bounded, up to a constant, by the integrable function $\fom_0$, being $\fom_0(x)  q_\delta (x,y) \le \fom_0(x)/(2\pi (1-\rme^{-2\delta}))^{d/2}$. 
Now, recall that
\cite[Theorem 6.6.1]{bogachev2022fokker} and \cite[Theorem 9.4.3]{bogachev2022fokker} imply that the Fokker-Planck equation 
    \begin{equation}\label{kolmo_fom}
        \partial_t p_t(x) -\div (x  p_t) -\Delta p_t(x) =0 \quad \text{ for } \quad (t,x) \in  [0,T]\times\rset^d\eqsp,
    \end{equation}
    with initial condition $p_0 \in \rmC(\rset^d)$, has a unique solution $p : [0,T] \times \rset^d$ such that $p_t\in\mrC^{1,2}((0,T]\times \rset^d)$.
    However, by \cite[Lemma 2.4]{figalli2008existence} and \cite[Theorem 8.1.1]{stroock1997multidimensional},
$(\fom_t)_{t \in [\delta,T]}$ is the (unique) weak solution to \eqref{kolmo_fom}  with initial condition $\fom_{\delta}$ which have shown to be continuous. The proof is therefore complete.

\subsection{Proof of \Cref{gYup}}
\label{sec:proof-crefgyup}
    Recall that the transition density $q_t(x,y)$ associated to the Ornstein-Uhlenbeck semigroup of \eqref{FE} is given by
\eqref{trans_dens_OU}. It then follows from the dominated convergence theorem that if $(P_t)_t$ is the semigroup associated to \eqref{FE}, then for any $f\in \rml^1(\gamma^d)$ and $t\in [0,T]$ it holds
\begin{equation}\label{inequality_P}
\begin{split}
\norm{\nabla P_t f(x)} &= \norm{\nabla\mathbb{E}[f(x\rme^{-t} + \sqrt{1-\rme^{-2t}}Z)]}= \rme^{-t}\norm{\mathbb{E}[\nabla f(x\rme^{-t} + \sqrt{1-\rme^{-2t}}Z)]}\\
&= \rme^{-t} \norm{P_t \nabla f}\le  \rme^{-t} P_t \norm{\nabla f}\eqsp.
\end{split}
\end{equation}Moreover, being $(P_t)_t$ self-adjoint, for any $\mu, \nu \in \mathcal{P}(\rset^d)$ with $\mu \ll \nu$ it holds 
\\$\rmd \mu P_t/ \rmd \nu= P_t \rmd \mu/\rmd \nu\eqsp.$
Using such properties of $(P_t)_t$ plus Jensen inequality applied to $\psi : \rset^d \times \rset_+ \rightarrow \rset_+\eqsp,  (x, \eta) \mapsto \norm{x}^2/\eta$, which is convex (see e.g.,  \cite[Proposition 2.3]{combettes2017perspective}), plus the fact that $\gamma^d P_t= \gamma^d$ for any $t\in [0,T]$, we infer the thesis: for any $0\le s\le t\le T-\delta$ it holds
\begin{equation}
    \begin{split}
        \gY(s)&= \int \norm{\nabla \log \frac{ \rmd \mustar P_{T-s}}{\rmd \gamma^d}}^2 \rmd \mustar P_{T-s}= \int \norm{\nabla \frac{ \rmd \mustar P_{T-s}}{\rmd \gamma^d}}^2 \frac{1}{ \rmd \mustar P_{T-s}/\rmd \gamma^d} \rmd \gamma^d\\
        &= \int \norm{\nabla P_{t-s} \frac{ \rmd \mustar P_{T-t}}{\rmd \gamma^d}}^2 \frac{1}{ P_{t-s}(\rmd \mustar P_{T-t}/\rmd \gamma^d)} \rmd \gamma^d\\
        &=\rme^{-2(t-s)} \int \norm{ P_{t-s} \nabla \frac{ \rmd \mustar P_{T-t}}{\rmd \gamma^d}}^2 \frac{1}{ P_{t-s}(\rmd \mustar P_{T-t}/\rmd \gamma^d)} \rmd \gamma^d\\
        &= \rme^{-2(t-s)} \int \psi \left( P_{t-s} \nabla \frac{ \rmd \mustar P_{T-t}}{\rmd \gamma^d}, P_{t-s} \frac{ \rmd \mustar P_{T-t}}{\rmd \gamma^d}\right) \rmd \gamma^d\\
        &\le \rme^{-2(t-s)} \int \psi \left( \nabla \frac{ \rmd \mustar P_{T-t}}{\rmd \gamma^d},  \frac{ \rmd \mustar P_{T-t}}{\rmd \gamma^d}\right) \rmd  \gamma^d P_{t-s} \\
        &= \rme^{-2(t-s)} \int \psi \left( \nabla \frac{ \rmd \mustar P_{T-t}}{\rmd \gamma^d},  \frac{ \rmd \mustar P_{T-t}}{\rmd \gamma^d}\right) \rmd  \gamma^d \\
        &=\rme^{-2(t-s)} \int \norm{\nabla \frac{ \rmd \mustar P_{T-t}}{\rmd \gamma^d}}^2 \frac{\rmd \gamma^d}{ \rmd \mustar P_{T-t}/\rmd \gamma^d} \\
        &= \rme^{-2(t-s)} \int \norm{\nabla \log \frac{ \rmd \mustar P_{T-t}}{\rmd \gamma^d}}^2 \rmd \mustar P_{T-t}= \rme^{-2(t-s)}\gY(t)\eqsp.
    \end{split}
\end{equation}

\subsection{Proof of \Cref{null_E}}
\label{sec:proof-crefnull_e}
If we show that for any $s\in [0, T-\delta],$ it holds\\
$
    \;\mathbb{E}[ \norm{Y_s \cdot Z_s}^2 ] <C, 
    $ for some constant $C>0$ independent of time, then, by  Fubini's theorem and \cite[Theorem 7.3]{baldi},
we are done. 

To do so, by the Cauchy-Schwarz inequality, we just need to show that $\mathbb{E}[ \norm{Y_s}^4 ]\eqsp$ and $ \mathbb{E}[ \norm{Z_s}^4 ]$ are bounded from above by constants which are independent of time $s \in \ccint{0,T-\delta}$. To this aim, recall that if $(\foX_t)_t$ is a strong solution of \eqref{FE}, then for any $s\in [0,T]$ the following equality holds in law
\begin{equation}
  \label{eq:OU_prob}
    \foX_s=\foX_0 \rme^{-s} + \sqrt{1-\rme^{-2t}} Z\eqsp, \quad \text{with} \quad Z \sim \gamma^d\eqsp.
\end{equation}Consequently, if $\mustar$ has bounded eighth-order moment, for any $s\in [0,T]$ the variable $\foX_s$ has eighth-order (hence fourth-order) moment uniformly bounded in time. Also, recall that the transition density $q_t(x,y)$ associated to the Ornstein-Uhlenbeck semigroup of \eqref{FE} is given by \eqref{trans_dens_OU}. But then, because of the dominated convergence theorem,  for any $s\in [\delta,T]$ it holds 
\begin{align}
  &  \nabla \log \fom_s(x)= \frac{1}{\fom_s(x)} \int \fom_0(y) \nabla_x q_s(y,x) \rmd y =  \frac{1}{\fom_s(x)} \int \fom_0(y) \frac{x-\rme^{-s}y}{\sigma_t^2} q_s(y,x)\rmd y\\
\label{eq:nabla_log_cond_esp}  
    & \nabla \log \fom_s(\foX_s) = \sigma_s^{-2}\mathbb{E}[\foX_s-\rme^{-s}\foX_0| \foX_s]\eqsp,
    \end{align}
  where $\sigma_s^2 = 1-\rme^{-2s}$.
Putting these results together and using Jensen inequality, we obtain that for any $t\in [0,T-\delta]$ it holds 
\begin{align}
  \norm{\nabla \log \fom_{T-t}(\foX_{T-t})}^4&= \sigma_{T-t}^{-8} \norm{\mathbb{E}[\foX_{T-t}-\rme^{-(T-t)}\foX_0| \foX_{T-t}]}^4\\
  &\le  \sigma_{T-t}^{-8} \mathbb{E}\left[\norm{\foX_{T-s}-\rme^{-(T-t)}\foX_0}^4\Big| \foX_{T-t}\right]\\
       &\lesssim \sigma_{T-t}^{-8}  \mathbb{E}\left[\norm{\foX_{T-t}}^4+ \rme^{-(T-t)}\norm{\foX_0}^4\Big| \foX_{T-t}\right]\eqsp.
    \end{align}
    From which it follows that  $\sup_{t\in[0,T-\delta]} \mathbb{E}[ \norm{Y_t}^4 ]< \plusinfty$.

    Similarly, for any $t\in [0,T-\delta]$ it holds

\begin{align}
    \nabla^2 \log \fom_s(x)&= \frac{1}{\fom_s(x)}\int \fom_0(y) \frac{(x-\rme^{-s}y)\cdot (x-\rme^{-s}y)}{(1-\rme^{-2s})^2} q_s(y,x)\rmd y \\
    &\quad + \frac{1}{\fom_s(x)} \int \fom_0(y) \frac{1}{1-\rme^{-2s}} q_s(y,x)\rmd y  \\
    &\quad - \frac{\nabla \fom_s(x)}{\fom_s(x)^2}\int \fom_0(y) \frac{x-\rme^{-s}y}{(1-\rme^{-2s})^2} q_s(y,x)\rmd y \\
        \nabla^2 \log \fom_s(\foX_s) &= \sigma_s^{-4} \mathbb{E}[(\foX_s-\rme^{-s}\foX_0)\cdot (\foX_s-\rme^{-s}\foX_0) | \foX_s]+1 \\
    &\quad - \sigma_s^{-4}\mathbb{E}[\foX_s-\rme^{-s}\foX_0| \foX_s] \cdot \mathbb{E}[\foX_s-\rme^{-s}\foX_0| \foX_s]\eqsp.
\end{align}
Proceeding as before, we therefore get that $\sup_{t\in[0,T-\delta]} \mathbb{E}[ \norm{Z_t}^4 ]< \plusinfty$ and the proof is completed.

\subsection{Proof of \Cref{theo_exponential_then_linear}}\label{appendix_exponential_case}We preface this proof by the following result, whose proof is postponed to \Cref{sec:proof-crefcd0n}.
\begin{figure}
\centering
\includegraphics[width=0.5\textwidth]{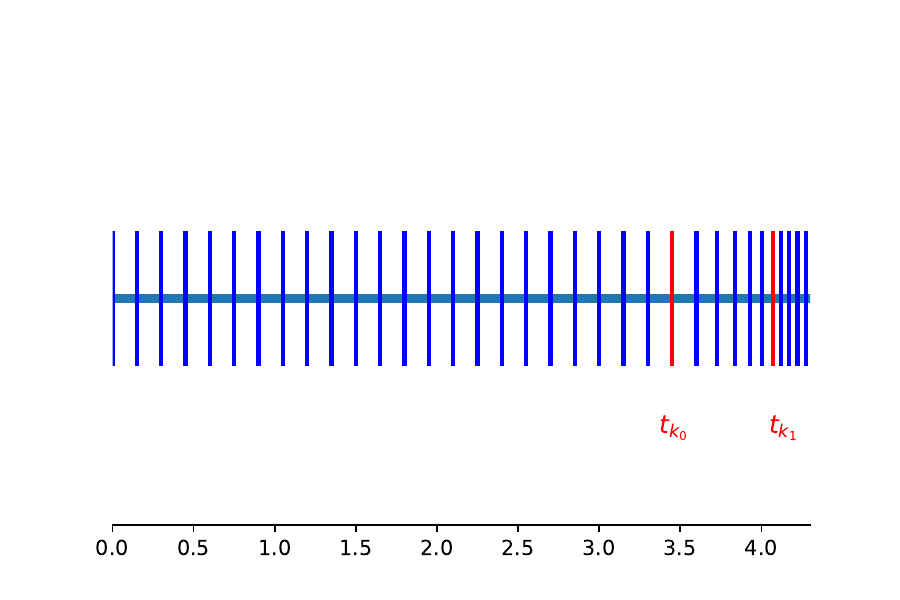}
\caption{Example of sequences $(h_k)_{k\in\{1,\ldots,N\}}$ and $(t_k)_{k\in\{1,\ldots,N\}}$ for $T= 4+ 2c$, $c=0.15$ and $a = 1/3$.}
\label{fig:exmapl_stepsize}
\end{figure}

\begin{lemma}\label{CD0N}
It holds
\begin{equation}
\mathbb{E}[\norm{Y_t}^2] \lesssim \frac{d}{1-\rme^{-2(T-t)}} +\Mtt^2_2 +d \eqsp.
\end{equation}
\end{lemma}
We start by giving a more explicit expression for the sequence of step-sizes.
For fixed $T,a,c$, we first define an explicit sequence $({h}_k)_{k\in\{1,\ldots,N\}}$ (and as a result, $({t}_k)_{k\in\{1,\ldots,N\}}$) satisfying for $k < N$, $h_{k+1}= c \min\{ \max\{T-t_k,a\},1\}$.
\begin{enumerate}[label=$\bullet$]
\item Set $k_0=\max\{k \geq 0 \, : \, T- {t}_k \geq 1\}$ and for $k \in \{1,\ldots,k_0\}$ define ${h}_k = c$ and
  ${t}_k = ck$.
\item Set  $k_1=\max\{k \geq 0 \, : \,T-{t}_{k+k_0} \geq a\}$ and for $k \in\{k_0+1,\ldots,k_0+k_1\}$ define
  $h_{k+1} = c(T-t_k)$ and $t_{k+1} = t_k+h_{k+1}$ .
\item Set $k_2 = \max \{ k \geq 0 \, :\, T - {t}_{k+k_0+k_1} \geq 0 \}$ and for $k \in \{k_0+k_1+1,\ldots,k_0+k_1+k_2\}$ define   $h_{k+1} = ca $ and $t_{k+1} = t_k+h_{k+1}$.
\item Finally, set   $N = k_0+k_1+k_2+1$ and define $h_N = T-t_{N-1}$ so that $t_N = T$.  
\end{enumerate}
To summarize
\begin{equation}\label{step_size_dec}
  h_{k+1} = \begin{cases}
T-t_{N-1} \quad & \mbox{$k = N-1$}\\
    c a  \quad & \mbox{$k_0+k_1 \le k \le k_0+k_1+k_2-1$}\\
                    c(T-t_k)  & \mbox{$k_0\leq k\leq k_0+k_1-1$}\\
                    c         & \mbox{$0\leq k\leq k_0-1$} \eqsp.
\end{cases}
\end{equation}
An example of such sequences are provided in \Cref{fig:exmapl_stepsize}.
We show next that $k_0,k_1,k_2$ and $N$ are well-defined but we can already notice that if so, then since $T \geq 1+2c$, distinguishing the four cases, we have  that for $k < N$, $h_{k+1}= c \min\{ \max\{T-t_k,a\},1\}$.
We now show that
\begin{equation}
  \label{eq:bound_k_0}
\begin{aligned}
  &k_0 =  \floor{c^{-1}(T-1)} \eqsp, \quad k_1 = \floor{\log{(a/(T-t_{k_0}))}/\log(1-c)} \lesssim \log(1/a)/c \eqsp, \\
   &  N-k_0-k_1 = k_2+1 \lesssim 1/ c \eqsp.
\end{aligned}
\end{equation}
By definition, $t_k= ck$ for $k \in \{0,\ldots,k_0\}$, and therefore $k_0 = \floorLigne{c^{-1}(T-1)} \geq 2$ since $T \geq 1+2c$, and $T-t_{k_0} \leq 1+c \leq 2$. In addition, for $k \in \{0,\ldots,k_1-1\}$, we have
\begin{equation}
  \label{eq:step_size_dec_2}
  h_{k_0+k+1} = c(T-t_{k_0+k})\eqsp,
\end{equation}
and therefore a simple recursion shows that $h_{k_0+k+1} = c(1-c)^{k}(T-t_{k_0})$ which implies that $k_1 =\max\{k \in \{0,\ldots,N-k_0\}\, : \,T-t_{k+k_0} \ge a\} = \max\{k \in \{0,\ldots,N-k_0\}\, : \, c^{-1}h_{k+k_0+1} \ge a\} = \floor{\log{(a/(T-t_{k_0}))}/\log(1-c)}$. It remains to show the last inequality in \eqref{eq:bound_k_0} which easily follows from \eqref{eq:bound_k_0}-\eqref{eq:step_size_dec_2}, $T-t_{k_0+k_1} \leq a/(1-c) \lesssim a$ and 
\begin{equation}
  \label{eq:6}
  k_0 c + \sum_{k=1}^{k_1} h_{k_0+k} +k_2c a = t_{k_0+k_1} + k_2 c a   \leq T \eqsp.
\end{equation}

The proof follows the same scheme as the proof of \Cref{theo_absolute_OU_just_fisher}, the only difference being the way we handle the error term $E_3$. In this context, this term rewrites as 
\begin{equation}
E_3 \lesssim \sum_{k=0}^{N-1}\int_{t_k}^{t_{k+1}}\mathbb{E}[\norm{Y_t-Y_{t_k}}^2]\rmd t \eqsp.
\end{equation}
Using the same approximation technique as in \Cref{lastmin_lemma}, it is sufficient to consider the case where $\mustar$ admits eighth-order moments.

Arguing on the basis of It\^o's formula exactly as we did in the proof of \Cref{lastmin_lemma} (see \eqref{eq:need_ineq_later_dec_step}) we find that for all $t\in[t_k,t_{k+1}]$
\begin{equation}
\mathbb{E}[\norm{Y_t-Y_{t_k}}^2] \lesssim g(t_{k+1})-g(t_k)
\end{equation}
implying
\begin{equation}\label{E3_exponential}
E_3\lesssim \sum_{k=0}^{N-1}h_{k+1} \{g(t_{k+1})-g(t_k)  \}\eqsp.
\end{equation}
The sum on the RHS can be rewritten as
\begin{align}
  \sum_{k=0}^{N-1}h_{k+1} (g(t_{k+1})-g(t_k)) &= g(T)h_N -g(0)h_1+\sum_{k=1}^{N-1}g(t_k)(h_k-h_{k+1}) \\
  & \qquad \lesssim g(T)h_N+\sum_{k=1}^{N-1}g(t_k)(h_k-h_{k+1}) \eqsp.
\end{align}
Using \eqref{step_size_dec}, $g$ is increasing by \Cref{PMP}, $h_{k_0+k}-h_{k_0+k+1} = c h_{k_0+k}$ by \eqref{eq:step_size_dec_2} for $k \in\{1,\ldots,k_1\}$, we get

\begin{equation}
\begin{split}
g(T)h_N+\sum_{k=1}^{N-1}g(t_k)(h_k-h_{k+1}) &= (c-c(T-t_{k_0+1}))g(t_{k_0})+ c\sum_{k=1}^{k_1}g(t_{k_0+k})h_{k_0+k}\\
&+cg(t_{k_0+k_1})(T-t_{k_0+k_1}-a)+ g(T) h_N\\
&+ g(t_{N-1})(h_{N-1}-h_N)\\
&\lesssim (c-c(T-t_{k_0+1}))g(t_{k_0})+ c\sum_{k=1}^{k_1}g(t_{k_0+k})h_{k_0+k}\\
&+cg(t_{k_0+k_1})(T-t_{k_0+k_1}-a)+ g(T) ca\eqsp.
\end{split}
\end{equation}
Let us now bound all the four terms on the RHS one-by-one. We begin with
\begin{equation}\label{eq:first_term}
\begin{split}
    (c-c(T-t_{k_0+1}))g(t_{k_0}) &\leq cg(t_{k_0})\\
    &\lesssim c g(t_{k_0}) \stackrel{\text{\Cref{CD0N}}}{\lesssim}  c \parenthese{\frac{d}{T-t_{k_0}}+\Mtt^2_2+d } \leq c(d+\Mtt^2_2)\eqsp.
\end{split}
\end{equation}
For the second term, we get by \eqref{step_size_dec} and \eqref{eq:bound_k_0},
\begin{equation}\label{eq:second_term}
\begin{split}
c\sum_{k=1}^{k_1}g(t_{k_0+k})h_{k_0+k} & \stackrel{\text{\Cref{CD0N}}}{\leq} c\sum_{k=1}^{k_1}\parenthese{\frac{d}{T-t_{k_0+k}}+\Mtt^2_2}h_{k_0+k} 
\lesssim c^2(d+\Mtt^2_2)k_1\\
&\lesssim c (d+\Mtt^2_2)\log(1/a)\eqsp .
\end{split}
\end{equation}
Recall that $ T-t_{k_0+k_1}\lesssim a$. As a result, we get 
\begin{equation}\label{eq:third_term}
\begin{split}
cg(t_{k_0+k_1})((T-t_{k_0+k_1})-a) 
&  \stackrel{\text{\Cref{CD0N}}}{\lesssim} c \Big(\frac{d}{T-t_{k_0+k_1}}+\Mtt^2_2+d\Big) (T- t_{k_0+k_1})\\
& \lesssim c(d+\Mtt^2_2)\eqsp.
\end{split}
\end{equation}
Finally, for the last term, we have by \eqref{eq:def_g} and definition of $\Ltt = g(T)/d$,
\begin{equation}\label{eq:fourth_term}
 c a g(T) = ca d\Ltt \eqsp.
\end{equation}
Plugging the bounds \eqref{eq:first_term}-\eqref{eq:second_term}-\eqref{eq:third_term}-\eqref{eq:fourth_term} back into \eqref{E3_exponential} gives \eqref{i-exponential_step}.

To prove the complexity bounds \eqref{TandNfinitefishabserr_exponential_step} it suffices to observe that choosing $T$ as in \eqref{TandNfinitefishabserr_exponential_step} and $$c=\frac{\varepsilon^2}{(d+\Mtt^2_2)\log \Ltt}$$ makes \eqref{i-exponential_step} of order $\tilde{O}(\varepsilon^2)$. To conclude it suffices to observe using \eqref{eq:bound_k_0} that the number of iterations $N$ is given by 

\begin{equation}
N=k_0 +k_1+(N-k_0-k_1) \lesssim \frac{T-1}{c}+\frac{\log(1/a)}{c}+\frac{1}{c} \lesssim \frac{\log(1/a)+T}{c}\eqsp.
\end{equation}

\subsection{Proof of \Cref{CD0N}}
\label{sec:proof-crefcd0n}
By \eqref{eq:nabla_log_cond_esp}, we have 
\begin{equation}
  \nabla \log \fom_s(\foX_s)=\frac{1}{1-e^{-2t}}\mathbb{E}[\foX_s-\rme^{-s}\foX_0| \foX_s]\eqsp,
\end{equation}
whence 
\begin{equation}
\nabla \log \ptilde_s(\foX_s)=\frac{1}{1-\rme^{-2t}}\mathbb{E}[\foX_s-\rme^{-s}\foX_0| \foX_s]+\foX_s \eqsp,
\end{equation}
and by \eqref{eq:OU_prob}
\begin{equation}
\begin{split}
\mathbb{E}[\norm{Y_{t}}^2]&\lesssim \frac{1}{(1-\rme^{-2(T-t)})^2}\mathbb{E}\parentheseDeux{\norm{\mathbb{E}[\foX_{T-t}-\rme^{-(T-t)}\foX_0| \foX_{T-t}]}^2}+2\mathbb{E}[\norm{\foX_{T-t}}^2]\\
&\lesssim \frac{1}{(1-\rme^{-2(T-t)})^2}\mathbb{E}\parentheseDeux{\norm{\foX_{T-t}-\rme^{-(T-t)}\foX_0}^2}+2\Mtt^2_2+2d \\
&=\frac{d}{(1-\rme^{-2(T-t)})} +2\Mtt^2_2 +2d \eqsp.
\end{split}
\end{equation}

\subsection{Proof of \Cref{technical_lemma}}\label{app_technical_lemma}
We start with \eqref{lim_RHS}. It trivially holds for $\Leb$-almost every $x$, $\rme^{-\norm{x}^2/n}\fom_0(x)\to \fom_0(x)$ as $n \to \plusinfty$ and, by the Lebesgue dominated convergence, it holds $\mfrZ_n \to 1$ too. It then follows that $\fom_0^n(x) \to \fom_0(x)$ $\Leb$-almost everywhere
      since
      \begin{equation}
        \label{eq:bound_diif_fom_0}
        \abs{\fom_0^n(x) - \fom_0(x)} \leq \{1-\rme^{-\norm{x}^2/n}\} \fom_0(x) + \abs{1-\mfrZ_n^{-1}} \fom_0(x) \eqsp.
      \end{equation}
       Moreover, since for  $\Leb$-almost every $x$, $\log (       \fom_0^n (x)) = \log\fom_0(x) - \norm{x}^2/n + C\eqsp,$ for some constant $C$, it holds $\nabla \log \left(\rme^{-\norm{x}^2/n}\fom_0(x)\right) = \nabla\log\fom_0(x) -2x/n\eqsp.$ As $2x/n \to 0$ $\Leb$-almost everywhere and in $\rml^2(\gamma^d)$, then $\nabla \log\fom_0^n \to \nabla\log\fom_0$ $\Leb$-almost everywhere and in $\rml^2(\gamma^d)$, hence \eqref{lim_RHS}  holds true.

      We proceed with \eqref{eq:conv_Y_n_Y}. To this end, we show first that: as $n\to +\infty$ we have 
      that for any $t \in (0,T]$ and $x \in\rset^d$, 
      \begin{equation}
        \label{lim_RHS_2}
\fom_t^n(x) \to \fom_t(x) \eqsp, \quad \nabla \log \fom_t^n(x) \to  \nabla \log \fom_t(x)  \eqsp.
      \end{equation}The proof of \eqref{lim_RHS_2} follows the same lines using the Lebesgue dominated convergence theorem, \eqref{eq:bound_diif_fom_0} and that
      \begin{equation}
        \label{eq:5}
        \fom_t^n(x) = \int q_t(y,x) \fom_0^n(y) \rmd y \eqsp, \quad         \fom_t(x) = \int q_t(y,x) \fom_0(y) \rmd y  \eqsp,
      \end{equation}
      and
      \begin{align}
        \label{eq:nabla_log_fom_n}
        \nabla \log   \fom_t^n(x) &= (1/\fom_t^n(x))\int \nabla_x \log q_t(y,x) \fom_0^n(y) q_t(y,x) \rmd y \eqsp,         \\
                \label{eq:nabla_log_fom}
        \nabla \log \fom_t(x) &=  (1/\fom_t(x))\int \nabla_x \log q_t(y,x) \fom_0(y) q_t(y,x) \rmd y \eqsp,
      \end{align}
      where $q_t$  is the transition density associated to the Ornstein-Uhlenbeck semigroup of \eqref{FE}  given by
\begin{equation}\label{trans_dens_OU_0}
q_t(x,y)=\frac{1}{(2\pi (1-\rme^{-2t}))^{d/2}}\exp\parenthese{-\frac{\norm{y-\mathrm{e}^{-t}x}^2}{2(1-\rme^{-2t})}}\eqsp.
\end{equation}

It follows directly from \cite[Theorem 11.3.4]{stroock1997multidimensional} and \eqref{lim_RHS_2} that $(\baX^n_t)_{t \in \ccint{0,T-\delta}} \Rightarrow (\baX_t)_{t \in \ccint{0,T-\delta}}$ as $n \to \plusinfty$ where $\Rightarrow$ denotes the convergence in distribution. We can now conclude the proof of \eqref{eq:conv_Y_n_Y}.

Note that to show \eqref{eq:conv_Y_n_Y}, it is sufficient to show that for any $t \in \ccint{0,T-\delta}$,
\begin{equation}
  \label{eq:conv_Y_n_Y_2}
\lim_{n \to \plusinfty}  \PE[\norm{Y_t^n - Y_t}\wedge 1] = 0 \eqsp.
\end{equation}
First by the triangle inequality, we have
\begin{align}
  \label{eq:8}
 & \PE[\norm{Y_t^n - Y_t}\wedge 1] \lesssim D_1^n + D_2^n \eqsp, \\
  & D_1^n= \expe{\norm{\nabla \log \fom^n_t(\baX_t^n) - \nabla \log \fom_t(\baX_t^n)}\wedge 1} \eqsp,\\
  &D_2^n = \expe{\norm{\nabla \log \fom_t(\baX_t^n) - \nabla \log \fom_t(\baX_t)} \wedge 1} \eqsp.
\end{align}
Using that $(\baX^n_t)_{t \in \ccint{0,T-\delta}} \Rightarrow (\baX_t)_{t \in \ccint{0,T-\delta}}$ as $n \to \plusinfty$ and \Cref{mapm_smooth}, we get that $\lim_{n \to \plusinfty} D_2^n = 0$. Finally, using the Lebesgue dominated convergence theorem, \eqref{lim_RHS}-\eqref{eq:nabla_log_fom_n}-\eqref{eq:nabla_log_fom}, we get  $\lim_{n \to \plusinfty} D_1^n = 0$ which completes the proof of \eqref{eq:conv_Y_n_Y_2}.

\subsection{Proof of \Cref{mapmG_smooth}}
\label{sec:proof-crefm-1}
Recall that given $t\in [0,T]$ and $u_1=(x_1, v_1), u_2=(x_2, v_2) \in \rset^{2d}$ the transition density $q_t(u_1,u_2)$ associated to the kinetic Ornstein-Uhlenbeck semigroup of \eqref{FEG} is given by
\begin{equation}\label{trans_dens_kOU}
q_t(u_1,u_2)=\frac{1}{(2\pi)^{d/2} \det(\Sigma_s)^{1/2}}\exp\parenthese{-\frac{\Sigma_s^{-1/2}(u_2-u_1 \rme^{-\bfA t}) \cdot (u_2-u_1 \rme^{-\bfA t})}{2}}\eqsp,
\end{equation} where 
\begin{equation}\label{obj_kOU}
    \bfA:=\begin{pmatrix}
0 & -\Id_d  \\
\Id_d & 2\Id_d
\end{pmatrix}\eqsp, \quad
 \Sigma_s=  \int_{0}^{s} \rme^{-(s-r)\bfA} \bfSigma \bfSigma^{\transpose} \rme^{-(s-r)\bfA^{\transpose}} \rmd r\eqsp,
\quad \bfSigma=\begin{pmatrix}
0  \\
2\Id_d
\end{pmatrix}\eqsp.
\end{equation}
But then, the proof is almost identical to that of \Cref{mapm_smooth} (just use the smoothness of the coefficients appearing in the Fokker-Planck equation satisfied weakly by $\fom_t(x,v)$, \eqref{trans_dens_kOU}, \cite[Theorem 6.6.1]{bogachev2022fokker} and \cite[Theorem 9.4.3]{bogachev2022fokker}).

\subsection{Proof of \Cref{kinetic_eight_moment}}
Exploiting \eqref{trans_dens_kOU} plus the fact that if \\$(\overrightarrow{U}_t)_t:=(\foX_t, \foV_t)_t$ is a strong solution of \eqref{FEG}, then for any $s\in [0,T]$ the following equality holds in law
\begin{equation}
    \overrightarrow{U}_s=\foX_0 \rme^{-As} + \sqrt{\Sigma_s} \; Z\eqsp, \quad \text{with} \quad Z \sim \gamma^{2d}\eqsp,
\end{equation}with $A$ and $\Sigma_s$ defined in \eqref{obj_kOU}, we can proceed as in the proof of \Cref{null_E} and obtain that for $t\in [0, T-\delta]$, $\mathbb{E}[\nabla_{(x,v)}\log \fom_{T-t}(\foX_{T-t}, \foV_{T-t})]$ and \\$\mathbb{E}[ \nabla^2_{(x,v)}\log \fom_{T-t}(\foX_{T-t}, \foV_{T-t})]$ are uniformly bounded in time. The thesis then follows from Holder inequality, Fubini's theorem and \cite[Theorem 7.3]{baldi}.

\subsection{Proof of \Cref{gup}}
\label{sec:proof-crefgup}
Recall that the transition density\\ $q_t(u_1,u_2)$ associated to the kinetic Ornstein-Uhlenbeck semigroup of \eqref{FEG} is given by
\eqref{trans_dens_kOU}. It then follows from the dominated convergence theorem that if $(P_t)_t$ is the semigroup associated to \eqref{FEG}, then for any $f\in \rml^1(\gamma^{2d})$ and $t\in [0,T]$ it holds
\begin{equation}\label{inequality_PG}
\begin{split}
\norm{\nabla P_t f(u)} &= \norm{\nabla\mathbb{E}[f(u\rme^{-\bfA t} + \sqrt{\Sigma_s}Z)]}= \norm{ \rme^{-\bfA t}}\norm{\mathbb{E}[\nabla f(u\rme^{-\bfA t} + \sqrt{\Sigma_s}Z)]}\\
&\lesssim (t+1) \rme^{-t} \norm{\mathbb{E}[\nabla f(u\rme^{-\bfA t} + \sqrt{\Sigma_s}Z)]}= (t+1) \rme^{-t} \norm{P_t \nabla f(u)}\\
&\lesssim \rme^{-t/4} P_t \norm{\nabla f(u)}\eqsp,
\end{split}
\end{equation}
where we used the fact that $\norm{\bfA \rme^{-\bfA t}}\lesssim (t+1) \rme^{-t} \lesssim \rme^{-t/4}$. But then, using the equivalence between the two norms $(x,v) \mapsto (\norm{x}^2 +\norm{v}^2)^{1/2}$ and $(x,v) \mapsto (\norm{x+v}^2 +\norm{v}^2)^{1/2}$, the proof follows the same lines as the proof of \Cref{gYup} and get the thesis. 
\section{Illustration of \Cref{ass:estimation-relative}} \label{Mean_of_H3}

Let $m\in \rset^d$ be unknown and assume $\mustar$ is the Gaussian distribution with mean $m$ and covariance matrix $\Id$. By taking
\begin{equation}
  \label{eq:7}
  \{s_\theta\}_{\theta\in \rset^{d\times N}}:= \left\{(t,x) \mapsto -2\sum_{k=1}^{N}(x-\theta_k) \Id_{\ocint{t_{k-1},t_{k}}}(t)\right\}_{\theta = (\theta_1,\ldots,\theta_N)\in \rset^{d \times N}}\eqsp,
\end{equation}
we show here that \Cref{ass:estimation-relative} holds with high probability.
Consider $(\foX_{i,t_k})_{i=1,...,M}$ \iid \; samples of $\foX_{T-t_k}$, then the empirical risk associated with these samples and with the map
\begin{equation}
  \label{eq:9}
  \theta \mapsto  \mathbb{E} \left[\norm{s_{\theta}(T-t_k,\foX_{t_k})- 2\nabla \log \ptilde_{T-t_{k}}(\foX_{t_k})}^2\right] \eqsp,
\end{equation}
for $k \in 0, ...,M-1$, 
is simply
\begin{equation}
  \label{eq:10}
  \theta \mapsto \frac{4}{M} \sum_{i=1}^{M} \norm{\theta_k - \foX_{i,t_k}}^2 \eqsp.
\end{equation}
As a result, the minimizer is given by the empirical mean
$\theta^\star_{k}:= \sum_{i=1}^{M} \foX_{i,t_k}/M$
and we therefore set $\theta^{\star} = (\theta_1^\star,\ldots,\theta_M^{\star})$.
Recall that if $\{Y_i\}_{i=1}^{M}$ are \iid~Gaussian random variables with mean $m$ and covariance matrix $\Id$, with probability $1-\eta$, it holds
\begin{equation}
  \label{eq:11}
  4\norm{\frac{1}{M} \sum_{i=1}^{M} Y_i - m }^2 \leq \delta_\eta^2\eqsp,
\end{equation}where 
\begin{equation}\label{eps_eta}
    \delta_\eta^2:=\frac{8d}{M}\log\left(\frac{2d}{\eta}\right)\eqsp.
\end{equation}
Indeed, denoting by $\{\mathbf{e}_i\}_{i=1}^d$ the canonical basis of $\rset^d$, for any $\lambda>0$, because of Markov inequality, it holds
\begin{equation}
\begin{split}
    &\PP\Big( \norm{\frac{1}{M} \sum_{i=1}^{M} Y_i - m }^2 > \frac{\delta_\eta^2}{4}\Big) \\
    &= \PP\Big( \norm{\frac{1}{M} \sum_{i=1}^{M} Y_i - m } > \frac{\delta_\eta}{2}\Big)\leq \PP \Big(\max_{j=1, ..., d}\Big\{ \mathbf{e}_j^{\transpose} \frac{1}{M} \sum_{i=1}^{M} Y_i - m \Big\}>\frac{\delta_\eta}{2\sqrt{d}}\Big)\\
    &\leq \sum_{j=1}^{d}\PP \Big(\mathbf{e}_j^{\transpose} \frac{1}{M} \sum_{i=1}^{M} Y_i - m >\frac{\delta_\eta}{2\sqrt{d}}\Big)\leq \sum_{j=1}^{d} \rme^{-\lambda \frac{\delta_\eta}{2\sqrt{d}}}\PE\Big[\rme^{\lambda \mathbf{e}_j^{\transpose} \frac{1}{M} \sum_{i=1}^{M}Y_i-m}\Big]\\
    &\leq 2 d \rme^{-\lambda \delta_\eta/2\sqrt{d}} \rme^{\lambda^2/2M}\eqsp.
\end{split}
\end{equation}

By choosing $\lambda=M\delta_\eta/2\sqrt{d}  $, one gets \eqref{eq:11}.
This result and the fact that \\$ \mathbb{E} \left[\norm{2\nabla \log \ptilde_{T-t_k}(\foX_{T-t_k})}^2 \middle| (\foX_{i,t_k})_{i=1,...,N} \right] =4d$, imply that with probability $1-\eta$, for any $k \in 0, ..., N-1$ it holds
\begin{multline}
  \label{eq:9}
  \mathbb{E} \left[\norm{s_{\thetas}(T-t_k,\foX_{T-t_k})- 2\nabla \log \ptilde_{T-t_{k}}(\foX_{T-t_k})}^2 \middle| (\foX_{i,t_k})_{i=1,...,N} \right] \\
  \leq \varepsilon^2_\eta  \mathbb{E} \left[\norm{2\nabla \log \ptilde_{T-t_k}(\foX_{T-t_k})}^2 \middle| (\foX_{i,t_k})_{i=1,...,N} \right] \eqsp,
\end{multline}with 
\begin{equation}
    \varepsilon_\eta^2:=\frac{\delta_\eta}{4d}=\frac{2}{M}\log\left(\frac{2d}{\eta}\right)\eqsp.
\end{equation}

\bibliographystyle{plain}
\bibliography{mybib}

@article{chen2022sampling,
place = {Country unknown/Code not available}, title = {Sampling is as easy as learning the score: theory for diffusion models with minimal data assumptions}, abstractNote = {}, journal = {International Conference on Learning Representations}, author = {Sitan Chen and Sinho Chewi and Jerry Li and Yuanzhi Li and Adil Salim and Anru R. Zhang},
year={2023}}

@inproceedings{chen2022improved,
  title={Improved analysis of score-based generative modeling: User-friendly bounds under minimal smoothness assumptions},
  author={Chen, Hongrui and Lee, Holden and Lu, Jianfeng},
  booktitle={International Conference on Machine Learning},
  pages={4735--4763},
  year={2023},
  organization={PMLR}
}

@article{pedrotti2023improved,
  title={Improved Convergence of Score-Based Diffusion Models via Prediction-Correction},
  author={Pedrotti, Francesco and Maas, Jan and Mondelli, Marco},
  journal={arXiv preprint arXiv:2305.14164},
  year={2023}
}

@book{bakry2014analysis,
  title={Analysis and geometry of Markov diffusion operators},
  author={Bakry, Dominique and Gentil, Ivan and Ledoux, Michel and others},
  volume={103},
  year={2014},
  publisher={Springer}
}

@article{nutz2021introduction,
  title={Introduction to entropic optimal transport},
  author={Nutz, Marcel},
  journal={Lecture notes, Columbia University},
  year={2021}
}

@inproceedings{cattiaux2021time,
  title={Time reversal of diffusion processes under a finite entropy condition},
  author={Cattiaux, Patrick and Conforti, Giovanni and Gentil, Ivan and L{\'e}onard, Christian},
  booktitle={Annales de l'Institut Henri Poincar{\'e} (B) Probabilit{\'e}s et Statistiques},
  volume={59},
  number={4},
  pages={1844--1881},
  year={2023},
  organization={Institut Henri Poincar{\'e}}
}

@article{anderson1982reverse,
  title={Reverse-time diffusion equation models},
  author={Anderson, Brian DO},
  journal={Stochastic Processes and their Applications},
  volume={12},
  number={3},
  pages={313--326},
  year={1982},
  publisher={Elsevier}
}

@article{van2014renyi,
  title={R{\'e}nyi divergence and Kullback-Leibler divergence},
  author={Van Erven, Tim and Harremos, Peter},
  journal={IEEE Transactions on Information Theory},
  volume={60},
  number={7},
  pages={3797--3820},
  year={2014},
  publisher={IEEE}
}

@book{bogachev2022fokker,
  title={Fokker--Planck--Kolmogorov Equations},
  author={Bogachev, Vladimir I and Krylov, Nicolai V and R{\"o}ckner, Michael and Shaposhnikov, Stanislav V},
  volume={207},
  year={2022},
  publisher={American Mathematical Society}
}

@article{figalli2008existence,
  title={Existence and uniqueness of martingale solutions for SDEs with rough or degenerate coefficients},
  author={Figalli, Alessio},
  journal={Journal of Functional Analysis},
  volume={254},
  number={1},
  pages={109--153},
  year={2008},
  publisher={Elsevier}
}

@book{stroock1997multidimensional,
  title={Multidimensional diffusion processes},
  author={Stroock, Daniel W and Varadhan, SR Srinivasa},
  volume={233},
  year={1997},
  publisher={Springer Science \& Business Media}
}

@book{durrett2019probability,
  title={Probability: theory and examples},
  author={Durrett, Rick},
  volume={49},
  year={2019},
  publisher={Cambridge university press}
}

@book{baldi,
  title={Stochastic calculus},
  author={Baldi, Paolo},
  year={2017},
  publisher={Springer}
}

@book{baldiD,
  title={Stochastic calculus},
  author={Baldi, Paolo},
  year={2017},
  publisher={Springer}
}

@article{salakhutdinov2015learning,
  title={Learning deep generative models},
  author={Salakhutdinov, Ruslan},
  journal={Annual Review of Statistics and Its Application},
  volume={2},
  pages={361--385},
  year={2015},
  publisher={Annual Reviews}
}

@article{ruthotto2021introduction,
  title={An introduction to deep generative modeling},
  author={Ruthotto, Lars and Haber, Eldad},
  journal={GAMM-Mitteilungen},
  volume={44},
  number={2},
  pages={e202100008},
  year={2021},
  publisher={Wiley Online Library}
}

@book{cedric2009hypocoercivity,
  title={Hypocoercivity},
  author={Villani, C{\'e}dric},
  volume={202},
  number={950},
  year={2009},
  publisher={American Mathematical Society}
}

@article{chiarini2021entropic,
  title={Entropic turnpike estimates for the kinetic Schr{\"o}dinger problem},
  author={Chiarini, Alberto and Conforti, Giovanni and Greco, Giacomo and Ren, Zhenjie},
  journal={Electronic Journal of Probability},
  volume={27},
  pages={1--32},
  year={2022},
  publisher={Institute of Mathematical Statistics and Bernoulli Society}
}

@article{goodfellow2020generative,
  title={Generative adversarial networks},
  author={Goodfellow, Ian and Pouget-Abadie, Jean and Mirza, Mehdi and Xu, Bing and Warde-Farley, David and Ozair, Sherjil and Courville, Aaron and Bengio, Yoshua},
  journal={Communications of the ACM},
  volume={63},
  number={11},
  pages={139--144},
  year={2020},
  publisher={ACM New York, NY, USA}
}

@inproceedings{kingma2014stochastic,
  title={Stochastic gradient VB and the variational auto-encoder},
  author={Kingma, Diederik P and Welling, Max},
  booktitle={Second international conference on learning representations, ICLR},
  volume={19},
  pages={121},
  year={2014}
}

@inproceedings{zhao2016energy,
title={Energy-based Generative Adversarial Networks},
author={Junbo Zhao and Michael Mathieu and Yann LeCun},
booktitle={International Conference on Learning Representations},
year={2017}
}

@article{papamakarios2021normalizing,
  title={Normalizing flows for probabilistic modeling and inference},
  author={Papamakarios, George and Nalisnick, Eric and Rezende, Danilo Jimenez and Mohamed, Shakir and Lakshminarayanan, Balaji},
  journal={The Journal of Machine Learning Research},
  volume={22},
  number={1},
  pages={2617--2680},
  year={2021},
  publisher={JMLRORG}
}

@inproceedings{turner2019metropolis,
  title={Metropolis-hastings generative adversarial networks},
  author={Turner, Ryan and Hung, Jane and Frank, Eric and Saatchi, Yunus and Yosinski, Jason},
  booktitle={International Conference on Machine Learning},
  pages={6345--6353},
  year={2019},
  organization={PMLR}
}

@article{dhariwal2021diffusion,
  title={Diffusion models beat gans on image synthesis},
  author={Dhariwal, Prafulla and Nichol, Alexander},
  journal={Advances in neural information processing systems},
  volume={34},
  pages={8780--8794},
  year={2021}
}

@article{song2019generative,
  title={Generative modeling by estimating gradients of the data distribution},
  author={Song, Yang and Ermon, Stefano},
  journal={Advances in neural information processing systems},
  volume={32},
  year={2019}
}

@inproceedings{NEURIPS2020_4c5bcfec,
 author = {Ho, Jonathan and Jain, Ajay and Abbeel, Pieter},
 booktitle = {Advances in Neural Information Processing Systems},
 editor = {H. Larochelle and M. Ranzato and R. Hadsell and M.F. Balcan and H. Lin},
 pages = {6840--6851},
 publisher = {Curran Associates, Inc.},
 title = {Denoising Diffusion Probabilistic Models},
 volume = {33},
 year = {2020}
}

@inproceedings{sohl2015deep,
  title={Deep unsupervised learning using nonequilibrium thermodynamics},
  author={Sohl-Dickstein, Jascha and Weiss, Eric and Maheswaranathan, Niru and Ganguli, Surya},
  booktitle={International conference on machine learning},
  pages={2256--2265},
  year={2015},
  organization={PMLR}
}

@article{de2021diffusion,
  title={Diffusion Schr{\"o}dinger bridge with applications to score-based generative modeling},
  author={De Bortoli, Valentin and Thornton, James and Heng, Jeremy and Doucet, Arnaud},
  journal={Advances in Neural Information Processing Systems},
  volume={34},
  pages={17695--17709},
  year={2021}
}

@inproceedings{
liu2022let,
title={Let us Build Bridges:  Understanding and Extending Diffusion Generative Models},
author={Xingchao Liu and Lemeng Wu and Mao Ye and qiang liu},
booktitle={NeurIPS 2022 Workshop on Score-Based Methods},
year={2022}
}

@inproceedings{benton2023linear,
  title={Nearly d-linear convergence bounds for diffusion models via stochastic localization},
  author={Benton, Joe and De Bortoli, Valentin and Doucet, Arnaud and Deligiannidis, George},
  booktitle={The Twelfth International Conference on Learning Representations},
  year={2024}
}

@article{de2022convergence,
title={Convergence of denoising diffusion models under the manifold hypothesis},
author={Valentin De Bortoli},
journal={Transactions on Machine Learning Research},
issn={2835-8856},
year={2022},
note={Expert Certification}
}

@article{block2020generative,
 author = {Block, Adam and Mroueh, Youssef and Rakhlin, Alexander},
 journal = {arXiv:2002.00107 [cs, math, stat]},
 month = {June},
 note = {arXiv: 2002.00107},
 title = {Generative Modeling with Denoising Auto-Encoders and Langevin Sampling},
 year = {2020}
}

@article{lee2022convergence,
  title={Convergence for score-based generative modeling with polynomial complexity},
  author={Lee, Holden and Lu, Jianfeng and Tan, Yixin},
  journal={Advances in Neural Information Processing Systems},
  volume={35},
  pages={22870--22882},
  year={2022}
}

@inproceedings{lee2023convergence,
  title={Convergence of score-based generative modeling for general data distributions},
  author={Lee, Holden and Lu, Jianfeng and Tan, Yixin},
  booktitle={International Conference on Algorithmic Learning Theory},
  pages={946--985},
  year={2023},
  organization={PMLR}
}

@inproceedings{
dockhorn2021score,
title={Score-Based Generative Modeling with Critically-Damped Langevin Diffusion},
author={Tim Dockhorn and Arash Vahdat and Karsten Kreis},
booktitle={International Conference on Learning Representations},
year={2022},
}

@inproceedings{
wibisono2022convergence,
title={Convergence in {KL} and R\'enyi Divergence of the Unadjusted Langevin Algorithm Using Estimated Score},
author={Kaylee Yingxi Yang and Andre Wibisono},
booktitle={NeurIPS 2022 Workshop on Score-Based Methods},
year={2022},
}

@article{hajlasz1996sobolev,
  title={Sobolev spaces on an arbitrary metric space},
  author={Haj{\l}asz, Piotr},
  journal={Potential Analysis},
  volume={5},
  pages={403--415},
  year={1996},
  publisher={Springer}
}

@article{rachev1985monge,
  title={The Monge--Kantorovich mass transference problem and its stochastic applications},
  author={Rachev, Svetlozar T.},
  journal={Theory of Probability \& Its Applications},
  volume={29},
  number={4},
  pages={647--676},
  year={1985},
  publisher={SIAM}
}

@article{carleo2019machine,
  title={Machine learning and the physical sciences},
  author={Carleo, Giuseppe and Cirac, Ignacio and Cranmer, Kyle and Daudet, Laurent and Schuld, Maria and Tishby, Naftali and Vogt-Maranto, Leslie and Zdeborov{\'a}, Lenka},
  journal={Reviews of Modern Physics},
  volume={91},
  number={4},
  pages={045002},
  year={2019},
  publisher={APS}
}

@article{brehmer2020madminer,
  title={MadMiner: Machine learning-based inference for particle physics},
  author={Brehmer, Johann and Kling, Felix and Espejo, Irina and Cranmer, Kyle},
  journal={Computing and Software for Big Science},
  volume={4},
  pages={1--25},
  year={2020},
  publisher={Springer}
}

@article{noe2019boltzmann,
  title={Boltzmann generators: Sampling equilibrium states of many-body systems with deep learning},
  author={No{\'e}, Frank and Olsson, Simon and K{\"o}hler, Jonas and Wu, Hao},
  journal={Science},
  volume={365},
  number={6457},
  pages={eaaw1147},
  year={2019},
  publisher={American Association for the Advancement of Science}
}

@article{sandfort2019data,
  title={Data augmentation using generative adversarial networks (CycleGAN) to improve generalizability in CT segmentation tasks},
  author={Sandfort, Veit and Yan, Ke and Pickhardt, Perry J and Summers, Ronald M},
  journal={Scientific reports},
  volume={9},
  number={1},
  pages={16884},
  year={2019},
  publisher={Nature Publishing Group UK London}
}

@article{ravuri2021skilful,
  title={Skilful precipitation nowcasting using deep generative models of radar},
  author={Ravuri, Suman and Lenc, Karel and Willson, Matthew and Kangin, Dmitry and Lam, Remi and Mirowski, Piotr and Fitzsimons, Megan and Athanassiadou, Maria and Kashem, Sheleem and Madge, Sam and others},
  journal={Nature},
  volume={597},
  number={7878},
  pages={672--677},
  year={2021},
  publisher={Nature Publishing Group UK London}
}

@article{ho2016generative,
  title={Generative adversarial imitation learning},
  author={Ho, Jonathan and Ermon, Stefano},
  journal={Advances in neural information processing systems},
  volume={29},
  year={2016}
}

@article{houthooft2016curiosity,
  title={Curiosity-driven Exploration in Deep Reinforcement Learning via Bayesian Neural Networks},
  author={Rein Houthooft and Xi Chen and Yan Duan and John Schulman and Filip De Turck and P. Abbeel},
  journal={ArXiv},
  year={2016}
}

@inproceedings{
song2020score,
title={Score-Based Generative Modeling through Stochastic Differential Equations},
author={Yang Song and Jascha Sohl-Dickstein and Diederik P Kingma and Abhishek Kumar and Stefano Ermon and Ben Poole},
booktitle={International Conference on Learning Representations},
year={2021}
}

@article{kusuoka2010existence,
  title={Existence of densities of solutions of stochastic differential equations by Malliavin calculus},
  author={Kusuoka, Seiichiro},
  journal={Journal of functional analysis},
  volume={258},
  number={3},
  pages={758--784},
  year={2010},
  publisher={Elsevier}
}

@inproceedings{follmer2005entropy,
  title={An entropy approach to the time reversal of diffusion processes},
  author={F{\"o}llmer, Hans},
  booktitle={Stochastic Differential Systems Filtering and Control: Proceedings of the IFIP-WG 7/1 Working Conference Marseille-Luminy, France, March 12--17, 1984},
  pages={156--163},
  year={2005},
  organization={Springer}
}

@article{durmus2015quantitative,
  title={Quantitative bounds of convergence for geometrically ergodic Markov chain in the Wasserstein distance with application to the Metropolis adjusted Langevin algorithm},
  author={Durmus, Alain and Moulines, {\'E}ric},
  journal={Statistics and Computing},
  volume={25},
  pages={5--19},
  year={2015},
  publisher={Springer}
}

@article{combettes2017perspective,
  title={Perspective functions: Properties, constructions, and examples},
  author={Combettes, Patrick L},
  journal={Set-Valued and Variational Analysis},
  volume={26},
  pages={247--264},
  year={2018},
  publisher={Springer}
}

@book{yong1999stochastic,
  title={Stochastic controls: Hamiltonian systems and HJB equations},
  author={Yong, Jiongmin and Zhou, Xun Yu},
  volume={43},
  year={1999},
  publisher={Springer Science \& Business Media}
}

@book{krylov2008controlled,
  title={Controlled diffusion processes},
  author={Krylov, Nikolaj Vladimirovi{\v{c}}},
  volume={14},
  year={2008},
  publisher={Springer Science \& Business Media}
}

@book{liptser2013statistics,
  title={Statistics of random Processes: I. general Theory},
  author={Liptser, Robert and Shiryaev, Albert N},
  volume={5},
  year={2013},
  publisher={Springer Science \& Business Media}
}

@inproceedings{tzen2019theoretical,
  title={Theoretical guarantees for sampling and inference in generative models with latent diffusions},
  author={Tzen, Belinda and Raginsky, Maxim},
  booktitle={Conference on Learning Theory},
  pages={3084--3114},
  year={2019},
  organization={PMLR}
}

@inproceedings{coleman2021sampling,
  title={Sampling, variational Bayesian inference, and conditioned stochastic differential equations},
  author={Coleman, Todd P and Raginsky, Maxim},
  booktitle={2021 60th IEEE Conference on Decision and Control (CDC)},
  pages={3054--3059},
  year={2021},
  organization={IEEE}
}

@inproceedings{lehec2013representation,
  title={Representation formula for the entropy and functional inequalities},
  author={Lehec, Joseph},
  booktitle={Annales de l'IHP Probabilit{\'e}s et statistiques},
  volume={49},
  number={3},
  pages={885--899},
  year={2013}
}

@incollection{karatzas2022variational,
  title={A Variational Characterization of Langevin-Smoluchowski Diffusions},
  author={Karatzas, Ioannis and Tschiderer, Bertram},
  booktitle={Stochastic Analysis, Filtering, and Stochastic Optimization: A Commemorative Volume to Honor Mark HA Davis's Contributions},
  pages={239--265},
  year={2022},
  publisher={Springer}
}

@article{fleming1985stochastic,
  title={Stochastic variational formula for fundamental solutions of parabolic PDE},
  author={Fleming, Wendell H and Sheu, Sheunn-Jyi},
  journal={Applied Mathematics and Optimization},
  volume={13},
  number={1},
  pages={193--204},
  year={1985},
  publisher={Springer}
}

@article{haussmann1986time,
  title={Time reversal of diffusions},
  author={Haussmann, Ulrich G and Pardoux, Etienne},
  journal={The Annals of Probability},
  pages={1188--1205},
  year={1986},
  publisher={JSTOR}
}

@article{pavon1989stochastic,
  title={Stochastic control and nonequilibrium thermodynamical systems},
  author={Pavon, Michele},
  journal={Applied Mathematics and Optimization},
  volume={19},
  pages={187--202},
  year={1989},
  publisher={Springer}
}

@inproceedings{oko2023diffusion,
  title={Diffusion models are minimax optimal distribution estimators},
  author={Oko, Kazusato and Akiyama, Shunta and Suzuki, Taiji},
  booktitle={International Conference on Machine Learning},
  pages={26517--26582},
  year={2023},
  organization={PMLR}
}


\hypersetup{
    colorlinks = false,
    linkbordercolor = {white},
    linkcolor = {white},
  }
  \textcolor{white}{\ref{eq:first_term}-\ref{eq:second_term}-\ref{eq:third_term}-\ref{eq:fourth_term} back into \ref{E3_exponential} gives \ref{i-exponential_step}.
    To prove the complexity bounds \eqref{TandNfinitefishabserr_exponential_step} it suffices to observe that choosing $T$ as in \eqref{TandNfinitefishabserr_exponential_step} and       We proceed with \eqref{eq:conv_Y_n_Y}, \eqref{eq:need_ineq_later_dec_step}
}

\end{document}